\documentclass[11pt]{amsart}
\usepackage{amsmath}
\usepackage{amsfonts}
\usepackage{amsthm}
\usepackage{amssymb}

\theoremstyle{plain} \numberwithin{equation}{section}
\newtheorem{Theorem}{Theorem}
\newtheorem{Lemma}[Theorem]{Lemma}
\newtheorem{Proposition}[Theorem]{Proposition}
\newtheorem{Corollary}[Theorem]{Corollary}
\newtheorem{Remark}[Theorem]{Remark}

\newtheorem{Criterion}[Theorem]{Criterion}
\newtheorem{Example}[Theorem]{Example}
\newtheorem{Question}[Theorem]{Question}

\theoremstyle{remark}

\begin{document}

\title[Asymptotic formulas]{Asymptotic formulas for spectral gaps
 and deviations of Hill and 1D Dirac operators}

{\author{Plamen Djakov}}\thanks{P. Djakov acknowledges the hospitality
of the Department of Mathematics of the Ohio State University,  July--August 2013.}

\author{Boris Mityagin}\thanks{B. Mityagin acknowledges the support
of the Scientific and Technological Research Council
of Turkey and the hospitality of Sabanci University, May--June, 2013.}

\address{Sabanci University, Orhanli,
34956 Tuzla, Istanbul, Turkey}
 \email{djakov@sabanciuniv.edu}

\address{Department of Mathematics,
The Ohio State University,
 231 West 18th Ave,
Columbus, OH 43210, USA} \email{mityagin.1@osu.edu}

\begin{abstract}
Let $L$ be the Hill operator or the one dimensional Dirac
operator on the interval $[0,\pi].$ If $L$ is considered with
Dirichlet, periodic or antiperiodic boundary conditions, then the
corresponding spectra are discrete and for large enough $|n|$  close
to $n^2 $ in the Hill case,  or close to $n, \; n\in \mathbb{Z}$ in
the Dirac case, there are one Dirichlet eigenvalue $\mu_n$ and two
periodic (if $n$ is even) or antiperiodic (if $n$ is odd) eigenvalues
$\lambda_n^-, \, \lambda_n^+ $ (counted with multiplicity).

We give estimates for the asymptotics of the spectral gaps $\gamma_n
= \lambda_n^+ - \lambda_n^-$ and deviations $ \delta_n =\mu_n -
\lambda_n^+$ in terms of the Fourier coefficients of the potentials.
Moreover, for special potentials that are trigonometric polynomials
we provide precise asymptotics of $\gamma_n$ and $\delta_n.$

\end{abstract}

\maketitle

{\it Keywords}: Hill operator, one-dimensional Dirac operator,
periodic,  antiperiodic and Dirichlet boundary conditions.

 {\it MSC:} 47E05, 34L40, 34L10.

\section*{Content}
\begin{enumerate}

\item[Section 1.]  Introduction\\

\item[Section 2.] Estimates of spectral gaps
and deviations\\

\item[Section 3.] Asymptotics of spectral gaps
$\gamma_n = \lambda_n^+ - \lambda_n^-$\\

\item[Section 4.] Asymptotics of deviations $\delta_n = \mu_n- \lambda_n^+$\\

\item[Section 5.]   Asymptotics of the spectral gaps and
deviations of Hill operators with potentials
$v(x) = a e^{-2ix} + b e^{2ix}$  \\

\item[Section 6.] Other two exponential term potentials\\

\item[Section 7.]   Asymptotics for Hill operators with potentials
$v(x) = a e^{-2ix} + b e^{2ix} + A e^{-4ix} + B e^{4ix}$ \\

\item[Section 8.]  Dirac operators with potentials
that are trigonometric polynomials

\item[Section 9.]  Concluding Remarks and Comments\\

\item[]   References

\end{enumerate}
\bigskip

\section{Introduction}

Let $L$ be either the Hill operator
\begin{equation}
\label{i1} L(y)=-y''+v(x)y,
\end{equation}
with a complex valued potential
$v \in L^2([0,\pi], \mathbb{C}),$    or the one dimensional
Dirac operator
\begin{equation}
\label{i2} L \, y = i
\begin{pmatrix} 1 & 0 \\ 0 & -1
\end{pmatrix}
\frac{dy}{dx}  + v(x) y, \quad y = \begin{pmatrix} y_1\\y_2
\end{pmatrix},  \quad
x\in[0,\pi],
\end{equation}
with a potential matrix $ v(x) =
\begin{pmatrix} 0 & P(x) \\ Q(x) & 0 \end{pmatrix}$ (not necessarily hermitian)
such that
$  P,Q \in L^2([0,\pi], \mathbb{C}).$
We consider these operators on the interval $[0,\pi] $ with
Dirichlet $(Dir),$  periodic  $(Per^+)$ and antiperiodic  $(Per^-)$
boundary conditions $(bc),$ and denote the corresponding closed
operators, respectively, by $L_{Dir}, $ $L_{Per^+}$ and $L_{Per^-}.$
Recall that in the Hill case
\begin{eqnarray}
\label{i3}  Dir: \qquad  y(0) = 0, \quad y(\pi) = 0, \\
\label{i4} Per^\pm: \quad   y(\pi) = \pm y(0), \quad y^\prime (\pi) =
\pm y^\prime (0),
\end{eqnarray}
and in the Dirac case
\begin{eqnarray}
\label{i5}  Dir: \qquad  y_1 (0) = y_2 (0), \qquad y_1 (\pi) = y_2 (\pi), \\
\label{i6} Per^\pm: \quad  y_1 (\pi) = \pm y_1 (0), \quad y_2 (\pi) =
\pm y_2 (0).
\end{eqnarray}

It is well known that the spectra of the operators $L_{Dir} $ and
$L_{Per^\pm} $ are discrete and for large enough $|n|$ close to $n^2
$ in the Hill case,  or close to $n, \; n\in \mathbb{Z}$ in the Dirac
case, there are one Dirichlet eigenvalue $\mu_n$ and two periodic (if
$n$ is even) or antiperiodic (if $n$ is odd) eigenvalues
$\lambda_n^-, \, \lambda_n^+ $ (counted with multiplicity). See
basics and details in \cite{E, MW,Mar,BESch, DM15}.

In this paper we study the ``asymptotic'' geometry of the
``spectral triangles'' with vertices $\lambda_n^-, \, \lambda_n^+,
\mu_n $ and provide asymptotic formulas for the {\em spectral gaps}
$\gamma_n = \lambda_n^+ - \lambda_n^-$ and {\em deviations}
$\delta_n = \mu_n - \lambda_n^+. $ Our study is motivated by the fact
that the rates of decay of spectral gaps and deviations are directly
related  to the smoothness of the potential  (see \cite{DM15, DM21})
and the Riesz basis
property of the root function systems of the operators $L_{Per^+}$
and $L_{Per^-}$  (see \cite{DM28,GT12} and the bibliography therein).

In the self-adjoint case (i.e., if $v$ is real valued in the Hill
case, or $Q(x)= \overline{P(x)}$ in the Dirac case), 
we have $\lambda_n^- \leq \mu_n \leq \lambda_n^+$
and the spectral
triangle reduces to a segment on the real line. In the context of Hill
operators Hochstadt \cite{Ho1,Ho2}  discovered a direct connection
between the smoothness of potentials $v$ and the rate of decay of
spectral gaps $\gamma_n = \lambda^+_n - \lambda^-_n:$ {\em If

$(A) \; \; v \in C^\infty, $ i.e., $v$ is infinitely differentiable,
then

$(B) \; \; \gamma_n $ decreases more rapidly than any power of $1/n.$

If a continuous function $v$ is a finite--zone potential, i.e.,
$\gamma_n =0$ for all large enough $n,$ then} $v \in C^\infty. $\\
In the middle of 70's (see \cite{MO1}, \cite{MT}) the latter
statement was extended, namely it was shown that  $ \; (B)
\Rightarrow (A)$ for real $L^2 ([0,\pi])$-potentials $v.$ Let us
mention also Trubowitz's result \cite{Tr} that a real valued
$L^2 ([0,\pi])$-potential $v$ is analytic if and only if $(\gamma_n )$
decays exponentially.

If $v$ is a complex-valued potential then the operator (\ref{i1}) is
non-self-adjoint, and in general the asymptotics of $\gamma_n
=\lambda_n^+ - \lambda_n^-$ does not determine the smoothness of $v$
as  Gasymov's example~\cite{Gas} showed. It was an idea of   Tkachenko
\cite{Tk92,Tk94} to bring into consideration the Dirichlet spectrum
and characterize the $C^\infty$-smoothness and analyticity in terms
of  $\gamma_n$ and $\delta_n= \mu_n - \lambda_n^+ $  (see also
\cite{ST}).

The above results have been obtained by using the Inverse Spectral
Theory. Kappeler and Mityagin \cite{KaMi01} suggested another
method, based on Fourier Analysis. They proved that if $\Omega=
(\Omega (m))_{m\in \mathbb{Z}} $ is a submultiplicative weight
sequence and the corresponding weighted Sobolev space is defined as
$$
H(\Omega) = \{v(x) = \sum_{k \in \mathbb{Z}} v_k e^{2ikx}, \quad
\sum_{k \in \mathbb{Z}} |v_k|^2 (\Omega(k))^2 < \infty \},
$$
then
 $$
 (A^*) \quad v\in H(\Omega) \quad \Longrightarrow \quad  (B^*) \quad
   (\gamma_n) ,  \left (\delta_n \right ) \in
  \ell^2(\Omega).
 $$
The same approach was used and developed further in \cite{DM3,DM5,
DM15} to obtain that $(B^*) \Rightarrow (A^*)$ under some mild
restrictions on the submultiplicative weight $\Omega. $

The analysis in \cite{KaMi01, DM3, DM5, DM15} is carried out under
the assumption $v\in L^2([0,\pi]).$  Following Savchuk and Shkalikov
\cite{SS03} and Hryniv and Mykytyuk \cite{HM01}, the authors
\cite{DM16} used quasi-derivatives  to develop a Fourier method for
studying the spectra of $L_{Per^\pm}$ and $L_{Dir}$ in the case of
periodic singular potentials. Moreover, the above results were
extended and strengthened as follows: if  $v \in H^{-1}_{per}
(\mathbb{R})$ and $\Omega$ is a weight of the form
$\Omega(m)=\omega(m)/|m|$ for $m \neq 0,$ with $\omega$ being a
sub-multiplicative weight, then $(A^*) \Rightarrow (B^*),  $ and
conversely, if in addition $(\log \omega (n))/n$ decreases to zero,
then $(B^*) \Rightarrow (A^*)$ (see Theorem 37 in \cite{DM21}.)

The situation is similar for Dirac operators but the relationship
between the smoothness of potentials and the decay rate of spectral
gaps has been considered much later \cite{GKM,GK1,DM7,DM6, DM15}.  Even
analogs of the Hochstadt and Trubowitz  criteria have been given
only in 2003 (see \cite{DM7}, where these results are announced, and
\cite{DM6} for complete proofs).

Again, in the self-adjoint case (i.e., when $\overline{Q}= P$) the
smoothness of $P$ and $Q$ is characterized by the decay rate of
$\gamma_n = \lambda_n^+ - \lambda_n^-, $ while in the
non-self-adjoint case this is not true, but the decay rate of the sequence
$|\gamma_n|  +|\delta_n|$ characterizes the smoothness of the
potential matrix (see Theorems 58 and 68 in \cite{DM15}).

The proofs of all these results use essentially the following statement
(see \cite[Section 2.4]{DM15} for Dirac operators with $L^2$-potentials
and \cite[Lemma 6]{DM21} for Hill-Schr\"odinger operators with
$H^{-1}_{per}$-potentials).

\begin{Lemma}
\label{basic}
There are functionals $\alpha_n (v;z) $ and $
\beta^\pm_n (v;z) $ defined for large enough $n \in \mathbb{N}$ and
$ |z| \leq n/4 $ in the Hill case, or for large enough $ |n|,  \; n
\in \mathbb{Z}$ and $ |z| \leq 1/2 $ in the Dirac case, such that
$\lambda = n^2 + z $ (or $\lambda= n+z $ in the Dirac case) is a
periodic (for even $n$) or anti-periodic (for odd $n$) eigenvalue of
$L$ if and only if $z$ is an eigenvalue of the matrix
\begin{equation}
\label{p1}  \left [
\begin{array}{cc} \alpha_n (v;z)  & \beta^-_n (v;z)
\\ \beta^+_n (v;z) &  \alpha_n (v;z) \end{array}
\right ].
\end{equation}
Moreover, $\alpha_n (z;v) $ and $\beta^\pm_n (z;v)$ depend
analytically on $v$ and $z,$  and $z_n^\pm =\lambda_n^\pm -n^2$
(respectively $z_n^\pm=\lambda_n^\pm -n$ in the Dirac case) are the
only solutions of the basic equation
\begin{equation}
\label{p2}   (z-\alpha_n (v;z))^2=  \beta^-_n (v;z)\beta^+_n (v;z),
\end{equation}
where $ |z| \leq n/4 $ in the Hill case, or $ |z| \leq 1/2 $ in the
Dirac case.
\end{Lemma}

The functionals $\alpha_n (v;z) $ and $\beta^\pm_n (v;z)$ are well
defined by explicit expressions in terms of the Fourier coefficients
of the potential (see \cite{DM15}, formulas (2.16)--(2.33) for Hill
operators, and formulas (2.71)--(2.80) for Dirac operators). In the
sequel, for convenience we suppress the dependence on $v$ in the
notations and write only $\beta^{\pm}_n (z), \alpha_n (z).$

The asymptotic behavior of $\beta^{\pm}_n (z)$ (or $\gamma_n$ and
$\delta_n$) plays also a crucial role in studying the Riesz
basis property of the system of root functions of the operators
$L_{Per^\pm}.$ In \cite[Section 5.2]{DM15}, it is shown (for
potentials $v \in L^2 ([0,\pi])$) that if the ratio $\beta^+_n
(z_n^*)/\beta^-_n (z_n^*) $ is not separated from $0$ or $\infty$
then the system of root functions of $L_{Per^\pm}$ does not contain
a Riesz basis (see Theorem 71 and its proof therein). Theorem 1 in
\cite{DM25} (or Theorem 2 in \cite{DM25a}) gives, for wide classes
of $L^2$-potentials,  the following criterion for Riesz basis
property.

\begin{Criterion}
\label{crit1}  Consider the Hill operator with  $v \in L^2
([0,\pi]). $ If
\begin{equation}
\label{a1} \beta_n^+ (0) \neq 0, \quad \beta_n^- (0)\neq 0
\end{equation}
and
\begin{equation}
\label{a2} \exists \, c  \geq 1 : \quad  c^{-1}|\beta_n^\pm (0)|
\leq |\beta_n^\pm (z)| \leq c \, |\beta_n^\pm (0)|, \quad  |z| \leq
1,
\end{equation}
for all sufficiently large even $n$ (if $bc= Per^+$) or odd $n$ (if
$bc= Per^-$), then

(a) there is $N=N(v) $ such that for $n>N$ the operator
$L_{Per^\pm}(v)$ has exactly two simple periodic (for even $n$) or
antiperiodic (for odd $n$) eigenvalues in the disc $D_n =\{z: |z-
n^2|<1 \} ;$

(b)  the system of root functions of $L_{Per^+} (v) $ or $L_{Per^-}
(v) $ contains a Riesz basis in $L^2 ([0,\pi])$ if and only if,
respectively,
\begin{equation}
\label{a3} \limsup_{n\in 2\mathbb{N}} t_n (0) <\infty \quad
\text{or}
 \quad \limsup_{n\in 1+2\mathbb{N}} t_n (0)<\infty,
\end{equation}
where
\begin{equation}
\label{a4} t_n (z)=\max \{|\beta_n^-(z)|/|\beta_n^+(z)|, \,
|\beta_n^+(z)|/|\beta_n^-(z)|\}.
\end{equation}
\end{Criterion}

In general form, i.e., without the restrictions (\ref{a1}) and
(\ref{a2}), this criterion is given in \cite{DM26-2} in the context
of 1D Dirac operators but the formulation and proof are the same in
the case of Hill operators (see Proposition~19 in \cite{DM28}).
Moreover, the same argument gives the following more general
statement.

\begin{Criterion}
\label{crit2} Let $\Gamma^+ =2\mathbb{N}, $ $\Gamma^-=2\mathbb{N}-1 $
in the case of Hill operators, and $\Gamma^+ =2\mathbb{Z}, $
$\Gamma^-=2\mathbb{Z}-1 $ in the case of one dimensional Dirac
operators. There exists $N_* = N_* (v)$ such that for $|n|>N_*$ the
operator $L=L_{Per^\pm}(v)$ has in the disc $D_n =\{z: |z- n^2|<n/4
\}$  (respectively $D_n =\{z: |z- n|<1/2 \}$) exactly two periodic
(for $n\in \Gamma^+$) or antiperiodic (for $n\in \Gamma^-$)
eigenvalues, counted with algebraic multiplicity. Let
    $$\mathcal{M}^\pm =\{n\in \Gamma^\pm: \; |n|
\geq N_*, \; \lambda^-_n \neq \lambda^+_n \},$$ and  let $\{u_{2n-1},
\,u_{2n} \}$ be a pair of normalized eigenfunctions associated, respectively,  with
the eigenvalues $\lambda_n^- $ and $ \lambda_n^+,$
$\; n \in \mathcal{M}^\pm.$

(a) If $\Delta \subset \Gamma^\pm, $ then the system $\{u_{2n-1},
\,u_{2n}, \; n\in \Delta \cap \mathcal{M}^\pm\} $ is a (Riesz) basis
in its closed linear span if and only if
\begin{equation}
\label{cr11} \limsup_{n\in \Delta \cap \mathcal{M}^\pm} t_n (z_n^*) <
\infty,
\end{equation}
where $z_n^* = \frac{1}{2} (\lambda^-_n + \lambda^+_n) -\lambda^0_n $
with $\lambda^0_n = n^2 $ for Hill operators and $\lambda^0_n = n $
for Dirac operators.

(b) The system of root functions of $L$ contains a Riesz basis if and
only if (\ref{cr11}) holds for $\Delta =\Gamma^\pm.$
\end{Criterion}

Another interesting abstract criterion of basisness is the following.

\begin{Criterion}
\label{crit3} The system of root functions of the operator
$L_{Per^\pm} (v)$ contains a Riesz basis in $L^2([0,\pi])$ if only
if
\begin{equation}
\label{a6} \limsup_{n \in \mathcal{M}^\pm} \frac{|\lambda_n^+-
\mu_n|}{|\lambda_n^+ - \lambda_n^-|} <\infty,
\end{equation}
where (for large enough $n$) $\mu_n $ is the Dirichlet eigenvalue
close to $n^2.$
\end{Criterion}

This criterion was given (with completely different proofs) in
\cite{GT12} for Hill operators with $L^2$-potentials and in
\cite{DM28} for Hill operators with $H^{-1}_{per}$-potentials and for
one-dimensional Dirac operators with $L^2$-potentials as well.

However, if one wants to apply Criterion~\ref{crit3} to specific
potentials $v,$ say $v(x) = a \cos 2x + b \cos 4x $ with $a,b \in
\mathbb{C},$ it is necessary first  to estimate the asymptotics of
the spectral gaps $|\gamma_n|= |\lambda_n^+ - \lambda_n^-|$ and
deviations $|\mu_n - \lambda_n^+|,$  what is by itself quite a
difficult problem.

In Section 2 we sharpen the known estimates for  the asymptotics of
the gap sequence $|\gamma_n|$ in terms of $|\beta_n^\pm (z_n^*)|$ --
see Proposition \ref{lem49} below -- and discuss the relationship
between Criterion 2 and Criterion 3.

Our goal in Section 3  and Section 4 is to find exact asymptotic
formulas for $\gamma_n $ and $\delta_n =\mu_n - \lambda_n^+$ in terms of
$\beta_n^\pm.$ Such formulas are derived under the assumption on
potentials that there is an infinite set $\Delta$ of indices $n$
such that $\beta_n^\pm (z_n^*) \neq 0$ and $\beta_n^\pm (z) \sim
\beta_n^\pm (z_n^*) $ if $|z-z_n^*| \leq 2(|\beta_n^- (z_n^*)|
+|\beta_n^+ (z_n^*)|).  $  \\
(Here and thereafter,
$ f(n) \sim g(n) $  means that $\lim f(n)/g(n) =1.$)

 A delicate analysis of the geometry of
eigenvectors corresponding to the eigenvalues $ \lambda_n^\pm $ and
$\mu_n $ for $n \in \Delta $  leads to the following asymptotic
formulas (see the precise claims in Theorems \ref{thm50} and \ref{thm51}):
If $n\in \Delta, $ then
$$
\gamma_n \sim 2\sqrt{\beta_n^- (z_n^*)} \sqrt{\beta_n^+ (z_n^*)}
\quad \text{for Hill and Dirac operators};
$$
$$
\mu_n - \lambda_n^+ \sim -\frac{1}{2}\left (\sqrt{\beta_n^-
(z_n^*)}+ \sqrt{\beta_n^+ (z_n^*)} \right )^2 \quad \text{for Hill operators};$$
$$
\mu_n - \lambda_n^+ \sim \frac{1}{2}\left (\sqrt{\beta_n^-
(z_n^*)}- \sqrt{\beta_n^+ (z_n^*)} \right )^2 \quad \text{for Dirac operators}.
$$
The above formulas are valid with appropriate choices of branches of
$\sqrt{\beta_n^- (z)}$ and $\sqrt{\beta_n^+ (z)}$ (depending on
$n$).  {\em All three formulas are new in the non-self-adjoint case.}

In the self-adjoint case, $\beta_n^- (z_n^*)= \overline{\beta_n^+
(z_n^*)}, $ so we obtain (see Corollary~\ref{cor14})
$$
\gamma_n \sim 2|\beta_n^+ (z_n^*)|  \quad \text{for Hill and Dirac operators},
$$
$$
\mu_n - \lambda_n^+ \sim -|\beta_n^+ (z_n^*)| - Re \left (\beta_n^+
(z_n^*) \right ) \quad \text{for Hill operators},
$$
$$
\mu_n - \lambda_n^+ \sim -|\beta_n^+ (z_n^*)| + Re \left (\beta_n^+
(z_n^*) \right ) \quad \text{for Dirac operators}.
$$
The formula for $\gamma_n $ is known (see \cite{DM3} for Hill
operators and \cite{DM6} for Dirac operators. However, to the best
of our knowledge the asymptotic formulas for $\mu_n -\lambda_n^+ $
{\em are new even in the self-adjoint case.}

In Section 5 and 6 we provide asymptotics for the spectral gaps and
deviations of Hill operators with potentials of the form
$v(x) = a e^{-2ix} + b e^{2ix} $   and
$v(x) = a e^{-2Rix} + b e^{2Six}. $

Section 7 is devoted to the more complicated case of Hill operators
with potentials $v(x) = a e^{-2ix} + b e^{2ix} + A e^{-4ix} + B e^{4ix}$
where $a,b, A, B $  are non-zero complex numbers.

Finally, in Section 8 we consider the Dirac operators with potentials
of the form
$v(x)= \begin{pmatrix}  0  &  P(x) \\ Q(x) &  0  \end{pmatrix}, \;\;$
$ P(x) =a e^{-2ix} + A e^{2ix}, \; Q(x) = b e^{-2ix} + B e^{2ix}$
and give formulas for their spectral gaps and
deviations.

\section{Estimates of spectral gaps and deviations}

In this section we estimate the absolute values of spectral gaps
$\gamma_n = \lambda_n^+ - \lambda_n^-$ and deviations
$\delta_n = \mu_n - \lambda_n^+$ for arbitrary potentials.
First, let us recall that $|\gamma_n |+ |\mu_n -\lambda_n^+|$ could
be estimated asymptotically in terms of $|\beta^-_n (z_n^*)|
+|\beta^+_n (z_n^*)|. $  By \cite[Theorem 66]{DM15} and
\cite[Theorem 37]{DM21}, the following holds.

\begin{Proposition}
\label{thm66} Let $L= L(v)$ be either the Hill-Schr\"odinger operator
with a $\pi$-periodic singular potential $v \in H^{-1}_{per}
(\mathbb{R}) $ or the one dimensional Dirac operator with a
$\pi$-periodic $L^2$-potential $v.$ For large enough $|n|,$ if $
\lambda^+_n , \lambda^-_n $ is the $n$-th couple of periodic (for
even $n$) or antiperiodic (for odd $n$)   eigenvalues of $L,$
$\gamma_n = \lambda^+_n - \lambda^-_n ,$ and $\mu_n $ is the $n$-th
Dirichlet eigenvalue of $L,$ then
\begin{equation}
\label{44.1}
 \frac{1}{144} \left ( |\beta^-_n (z_n^*)| +|\beta^+_n
(z_n^*)| \right ) \leq |\gamma_n | +|\mu_n - \lambda^+_n |  \leq 58
\left ( |\beta^-_n (z_n^*)| +|\beta^+_n (z_n^*)| \right ),
\end{equation}
where  $z^*_n = (z_n^+ +z_n^-)/2$ with $z_n^\pm = \lambda_n^\pm -
\lambda_n^0, $  $\lambda_n^0 = n^2 $ in the case of
Hill-Schr\"odinger operator and $\lambda_n^0 = n$ in the case of
Dirac operator.
\end{Proposition}

The following proposition is a modification of Lemma~49 in
\cite{DM15}. It gives a two-sided estimate for $|\gamma_n|$ in terms
of the expressions $t_n (z_n^*)$ defined in (\ref{a4}), while
Lemma~49 in \cite{DM15} gives only a lower estimate of $|\gamma_n| $
in terms of $t_n (z_n^+).$ The proof is similar but for convenience
of the reader we provide all details.

\begin{Proposition}
\label{lem49} Suppose the assumptions of Proposition~\ref{thm66}
hold. If the set $\mathcal{M} = \{n: \lambda_n^+ \neq \lambda_n^-
\}$ is infinite, then there exist $N^* \in \mathbb{N} $ and a
sequence of positive numbers $\eta_n \to 0 $ such that, for $n \in
\mathcal{M} $ with $|n|>N^*,$ we have
\begin{equation}
\label{32.01} \frac{2\sqrt{t_n}}{1+t_n} -\eta_n
 \leq
\frac{|\gamma_n |}{ |\beta^-_n (z_n^*)| +|\beta^+_n (z_n^*)|} \leq
\frac{2\sqrt{t_n}}{1+t_n} +\eta_n ,
\end{equation}
where
\begin{equation} \label{32.02} t_n = \begin{cases}
\max(|\beta_n^- (z_n^*)| /|\beta_n^+ (z_n^*)|, \, |\beta_n^+
(z_n^*)| /|\beta_n^- (z_n^*)|) \;\; &\text{if}\;\; \beta_n^-
(z_n^*)\beta_n^+ (z_n^*) \neq 0,\\
\infty \;\; &\text{if}\;\; \beta_n^- (z_n^*)\beta_n^+ (z_n^*) = 0
\end{cases}
\end{equation}
and $\frac{2\sqrt{t}}{1+t}=0$  for $t =\infty. $

\end{Proposition}

{\em Remark.} By (\ref{44.1}), if  $\beta_n^+ (z_n^*)=\beta_n^-
(z_n^*)=0,$ then $\lambda_n^+ = \lambda_n^-. $  Therefore, for $n
\in \mathcal{M},\;$   $|\beta^-_n (z_n^*)| +|\beta^+_n (z_n^*)|>0 $
for large enough $|n|.$

\begin{proof}
Let $\mathcal{D}_n = \{z: |z-z_n^*| \leq |\gamma_n|\},$ and let
\begin{equation}
\label{32.1} \epsilon_n= \sup_{\mathcal{D}_n} \left
|\frac{\partial \alpha_n}{\partial z } \right|+ \sup_{\mathcal{D}_n} \left |\frac{\partial
\beta_n^-}{\partial z }\right| + \sup_{\mathcal{D}_n} \left
|\frac{\partial \beta_n^+}{\partial z }\right|.
\end{equation}
Then $\epsilon_n \to 0$    (see \cite[Lemma 29]{DM15} for
Hill operators with $L^2$-potentials or  (4.32) in \cite[Proposition~15]{DM21} for
Hill operators with singular potentials, and (2.109) in
\cite[Proposition~35]{DM15} for Dirac operators with $L^2$-potentials).

Let $[w,z]$ denote the line path from $w$ to $z.$  Since
$$ \beta_n^\pm (w) - \beta_n^\pm (z_n^*) = \int_{[w,z_n^*]}
\frac{d}{d z} (\beta_n^\pm (z)  ) d z, $$
 (\ref{32.1}) implies that
$$ |\beta_n^\pm (w) - \beta_n^\pm (z_n^*) | \leq \epsilon_n
|w-z_n^*|  \quad \text{for} \quad w\in \mathcal{D}_n. $$ Therefore,
for $ w\in \mathcal{D}_n $ we have
\begin{equation}
\label{32.2} |\beta_n^\pm (z_n^*)| - \epsilon_n |w-z_n^*| \leq
 |\beta_n^\pm (w) | \leq |\beta_n^\pm
(z_n^*)| + \epsilon_n |w-z_n^*|.
\end{equation}
In an analogous way, by (\ref{32.1}) we infer that
\begin{equation}
\label{32.20} |\alpha_n(z_n^+) - \alpha_n (z_n^-) | \leq \epsilon_n
|z_n^+ -z_n^-|.
\end{equation}

Set
\begin{equation}
\label{32.21} \zeta_n^+  = z_n^+ - \alpha_n (z_n^+), \quad \zeta_n^-
= z_n^- - \alpha_n (z_n^-).
\end{equation}
By (\ref{32.20}), the triangle inequality implies that
\begin{equation}
\label{32.3} (1-\epsilon_n) |z_n^+ - z_n^-| \leq |\zeta_n^+ -
\zeta_n^-| \leq (1+ \epsilon_n) |z_n^+ - z_n^-|,
\end{equation}

On the other hand, by (\ref{p2}) we have
\begin{equation}
\label{32.3a} (\zeta_n^+)^2 = \beta_n^+ (z_n^+) \beta_n^- (z_n^+),
\quad (\zeta_n^-)^2 = \beta_n^+ (z_n^-) \beta_n^- (z_n^-),
\end{equation}
and therefore,
$$
|\zeta_n^\pm| = \sqrt{|\beta_n^+ (z_n^\pm) \beta_n^- (z_n^\pm)|}
\leq \frac{1}{2}(|\beta_n^+ (z_n^\pm)|+ |\beta_n^- (z_n^\pm)|).
$$
Taking into account that $\gamma_n = z_n^+ - z_n^-$ and $|z_n^\pm -
z_n^*| = |\gamma_n|/2,$ we obtain by (\ref{32.2}) and (\ref{32.3})
that
$$
(1-\epsilon_n) |\gamma_n| \leq |\zeta_n^+ - \zeta_n^-| \leq
|\zeta_n^+| + |  \zeta_n^-| \leq |\beta_n^+ (z_n^*)| + |\beta_n^-
(z_n^*)| + \epsilon_n |\gamma_n|.
$$
Since $\epsilon_n \to 0, $ it follows that for large enough $|n|$
\begin{equation} \label{32.4a}
 |\gamma_n| \leq 2(|\beta_n^+ (z_n^*)| + |\beta_n^-
(z_n^*)|).
\end{equation}

In view of (\ref{32.3a}), we have
\begin{equation}
\label{32.4}  (\zeta_n^+)^2 - (\zeta_n^-)^2 = \int_{[z_n^-,z_n^+]}
\frac{d}{dz} [\beta_n^+ (z) \beta_n^- (z)] dz.
\end{equation}
By (\ref{32.1}), (\ref{32.2}) and (\ref{32.4a}) it follows, for
$z\in [z_n^-,z_n^+],$ that
\begin{align}
\label{32.4b} \left | \frac{d}{dz} [\beta_n^+ (z) \beta_n^- (z)]
\right | &\leq \epsilon_n \left ( |\beta_n^+ (z_n^*)|+|\beta_n^-
(z_n^*)| + \epsilon_n |z_n^+ - z_n^- | \right )\\
 \nonumber &\leq
3\epsilon_n  \left ( |\beta_n^+ (z_n^*)|+|\beta_n^- (z_n^*)| \right
).
\end{align}
Since $\epsilon_n \to 0, \; $  (\ref{32.3}) and (\ref{32.4}) imply,
for large enough $|n|,$ that
\begin{align*}
|\zeta_n^+ +\zeta_n^-| \cdot |\zeta_n^+ - \zeta_n^-| &\leq
3\epsilon_n
 \left ( |\beta_n^+
(z_n^*)|+|\beta_n^- (z_n^*)| \right ) |z_n^+ - z_n^- | \\ &\leq
4\epsilon_n \left ( |\beta_n^+ (z_n^*)|+|\beta_n^- (z_n^*)| \right )
|\zeta_n^+ - \zeta_n^-|,
\end{align*}
so we conclude that for large enough $|n|$
\begin{equation}
\label{32.5} |\zeta_n^+ +\zeta_n^-| \leq 4\epsilon_n  \left (
|\beta_n^+ (z_n^*)|+|\beta_n^- (z_n^*)| \right ).
\end{equation}

On the other hand, from (\ref{32.4b}) it follows that
$$
|\beta_n^-(z_n^+)\beta_n^+(z_n^+) -
\beta_n^-(z_n^*)\beta_n^+(z_n^*)| \leq 3\epsilon_n
 \left ( |\beta_n^+
(z_n^*)|+|\beta_n^- (z_n^*)| \right ) |z_n^+ - z_n^*|.
$$
Since $|z_n^+ - z_n^*|= \frac{1}{2}|\gamma_n|, $   by (\ref{32.4a})
we obtain
\begin{equation}
\label{32.50} |\beta_n^-(z_n^+)\beta_n^+(z_n^+) -
\beta_n^-(z_n^*)\beta_n^+(z_n^*)| \leq  3\epsilon_n  \left (
|\beta_n^+ (z_n^*)|+|\beta_n^- (z_n^*)| \right )^2.
\end{equation}

In view of (\ref{32.02}) and (\ref{32.3a}), (\ref{32.50}) implies
\begin{equation}
\label{32.6}  \frac{t_n}{(1+t_n)^2}- 3\epsilon_n
 \leq
\frac{|\zeta_n^+|^2}{\left ( |\beta_n^+ (z_n^*)|+|\beta_n^- (z_n^*)|
\right )^2} \leq  \frac{t_n}{(1+t_n)^2}+ 3\epsilon_n,
\end{equation}
which leads to
\begin{equation}
\label{32.7}  \frac{\sqrt{t_n}}{1+t_n}- 2\sqrt{\epsilon_n}
 \leq
\frac{|\zeta_n^+|}{ |\beta_n^+ (z_n^*)|+|\beta_n^- (z_n^*)| } \leq
\frac{\sqrt{t_n}}{1+t_n}+ 2\sqrt{\epsilon_n}.
\end{equation}
Indeed, since $\sqrt{a+b} \leq \sqrt{a} + \sqrt{b}$ for $ a,b \geq
0,$ the right-hand inequality in (\ref{32.6}) implies the right-hand
inequality in (\ref{32.7}). If $\frac{\sqrt{t_n}}{1+t_n} \leq
2\sqrt{\epsilon_n},$ then the left-hand inequality in (\ref{32.7})
is trivial. If
 $\frac{\sqrt{t_n}}{1+t_n} >  2\sqrt{\epsilon_n},$ then
$$
\left (\frac{\sqrt{t_n}}{1+t_n}- 2\sqrt{\epsilon_n}\right )^2
=\frac{t_n}{(1+t_n)^2}- 4\sqrt{\epsilon_n}
\frac{\sqrt{t_n}}{1+t_n}+4\epsilon_n < \frac{t_n}{(1+t_n)^2}-
3\epsilon_n,
$$
so the left-hand inequality in (\ref{32.6}) implies the left-hand
inequality in (\ref{32.7}).

Since $$  2 |\zeta_n^+ | - |\zeta_n^+ + \zeta_n^- | \leq |\zeta_n^+
- \zeta_n^- | = |2 \zeta_n^+ - (\zeta_n^+ + \zeta_n^-)| \leq 2
|\zeta_n^+ | + |\zeta_n^+ + \zeta_n^- |, $$ from (\ref{32.5}) and
(\ref{32.7}) it follows, for large enough $n,$ that
\begin{equation*}
\label{32.8}  \frac{2\sqrt{t_n}}{1+t_n}- 8\sqrt{\epsilon_n}
 \leq
\frac{|\zeta_n^+- \zeta_n^-|}{ |\beta_n^+ (z_n^*)|+|\beta_n^-
(z_n^*)| } \leq \frac{2\sqrt{t_n}}{1+t_n}+ 8\sqrt{\epsilon_n}.
\end{equation*}
By (\ref{32.3}), the above inequalities imply (\ref{32.01}), say
with $\eta_n = 16\sqrt{\epsilon_n}+ 4 \epsilon_n.$

\end{proof}

Next we give a two-sided estimate of the ratio
\begin{equation}
\label{rn} r_n = \frac{|\mu_n- \lambda_n^+|}{|\gamma_n|}, \quad n \in
\mathcal{M},
\end{equation}
in terms of the numbers $t_n $ defined in (\ref{32.02}) under the
assumption that the set $\mathcal{M} = \{n: \; \gamma_n \neq 0\}$ is
infinite. Dividing (\ref{44.1}) by $|\gamma_n |$  and taking into
account (\ref{32.01}), we obtain,  for large enough $n\in
\mathcal{M},$ that
\begin{equation}
\label{32.11} \frac{1}{58} \left (\frac{2\sqrt{t_n}}{1+t_n} -\eta_n
\right )
 \leq
\frac{1}{ 1+r_n} \leq 144 \left (\frac{2\sqrt{t_n}}{1+t_n} +\eta_n
\right ), \quad \eta_n \to 0.
\end{equation}
Since  $t_n \geq 1, \; $      (\ref{32.11}) implies that
\begin{equation}
\label{r0} \frac{1}{58} \left (\frac{1}{\sqrt{t_n}} -\eta_n \right
)
 \leq
\frac{1}{ 1+r_n} \leq 144 \left (\frac{2}{\sqrt{t_n}} +\eta_n
\right ), \quad \eta_n \to 0,
\end{equation}
where $\frac{1}{\sqrt{t}} = 0 $  if $t=\infty.$

Suppose the sets
\begin{equation}
\label{r2} \mathcal{M}^+ = \{n \;\text{even}: \;\gamma_n \neq 0\},
\quad  \mathcal{M}^- = \{n \;\text{odd}: \;\gamma_n \neq 0\},
\end{equation}
are infinite, and set
\begin{equation}
\label{32.12} \tau^\pm = \limsup_{n\in \mathcal{M}^\pm} t_n, \quad
R^\pm = \limsup_{n\in \mathcal{M}^\pm}  r_n.
\end{equation}
By (\ref{r0}) we obtain that
$$
\frac{1}{58} \liminf_{n\in \mathcal{M}^\pm} \frac{1}{\sqrt{t_n}} \leq
\liminf_{n\in \mathcal{M}^\pm}  \frac{1}{1+r_n} \leq    288 \liminf_{n\in
\mathcal{M}^\pm} \frac{1}{\sqrt{t_n}}.
$$
Therefore, the following holds.
\begin{Proposition}
\label{prop21} In the above notations,  and under the assumptions of
Proposition~\ref{thm66}, we have
\begin{equation}
\label{r3} \tau^\pm = \infty \quad \Leftrightarrow \quad  R^\pm =
\infty
\end{equation}
and
\begin{equation}
 \label{32.14} \frac{1}{288}\sqrt{\tau^\pm} \leq 1+ R^\pm
\leq 58\sqrt{\tau^\pm} \quad \text{if} \;\; \tau^\pm <\infty.
\end{equation}

\end{Proposition}

{\em Remark.} Of course, Proposition~\ref{prop21} makes no sense if
the sets $\mathcal{M}^\pm$ are finite. If only one of these sets is
infinite,  say $\mathcal{M}^+$ is infinite but $\mathcal{M}^-$ is
finite, then (\ref{r3}) and (\ref{32.14}) hold for $\tau^+$ and
$R^+$ only.

\begin{Remark}
Proposition~\ref{prop21} shows directly why
Criterion~\ref{crit2} and Criterion~\ref{crit3} are equivalent
assertions (compare with Theorem~19   in \cite{DM28}, where this equivalence
follows from the fact that both criteria give necessary and sufficient conditions
for the root function system to contain Riesz bases).
\end{Remark}

Indeed, let us consider the case of periodic boundary
conditions. Suppose that the set $\mathcal{M}^+ $ is infinite. Then
we have $\tau^+ = \infty $ if and only if $R^+ = \infty,$ and by
Criterion~\ref{crit2} and Criterion~\ref{crit3} this is the case
when the system of periodic root functions does not contain a basis.
 On the other hand, (\ref{r3}) shows as well that $\tau^+ < \infty$ if and only
if $R^+ < \infty,$ and then the system of periodic root functions
contains a basis by the same criteria.
\bigskip

If $R^\pm < \infty,$ then in view of (\ref{32.12}) we have the
following localization of Dirichlet eigenvalues $\mu_n:$  for every
$ \eta > 0 $  there are $N^\pm (v)\in \mathbb{N}$ such that for
$n\in \mathcal{M}^+$ with $|n|
>N^+ (v)$ (in the case of periodic boundary conditions) or for
$n\in \mathcal{M}$ with $|n|
>N^- (v)$ (in the case of antiperiodic boundary conditions)
$$
|\mu_n - \lambda_n^+| \leq (R^\pm + \eta)|\gamma_n|.
$$
It is difficult to find the numbers  $R^\pm$
directly by (\ref{rn}) and (\ref{32.12})
but
Proposition~\ref{prop21}, (\ref{32.14}), gives a way to estimate
these numbers in terms of $\tau^\pm.$
\bigskip

Could one give some estimates on the localization of $\mu_n $ in the
case where $R^\pm=\infty? $ In the next section, we show  that the
answer to that question is positive under some additional assumptions
on the potential.

\section{Asymptotics of spectral gaps $\gamma_n = \lambda_n^+ - \lambda_n^-$}

Recall that close to $n^2 $ in the Hill case or close to $n$ in the Dirac case
there are exactly two periodic (for even $n$) or antiperiodic (for odd $n$)
eigenvalues (counted with multiplicity).
Here and thereafter, we denote by $\lambda_n^+ $ the eigenvalue with a
larger real part
or with a larger imaginary part if the real parts are equal.

Proposition \ref{lem49} in the previous section gives
estimates on the asymptotics of the sequence absolute values  $|\gamma_n |$
of spectral gaps.
Our goal in this section is to provide, for wide classes of potentials, an
asymptotic formula for the sequence of spectral gaps
$\gamma_n = \lambda_n^+ - \lambda_n^-$ in terms of $\beta_n^\pm (z_n^*). $

Our basic assumption on the
potentials is that there is an  infinite set $\Delta $ of even (if
$bc =Per^+$) or odd (if $bc =Per^-$) integers $n \in \mathbb{N} $ in the Hill case (or
$n \in \mathbb{Z}$ in the Dirac case) such that for $n \in \Delta$
\begin{equation}
\label{50.1}  \beta^-_n (z_n^*) \neq 0, \quad
 \beta^+_n (z_n^*) \neq 0
\end{equation}
and
\begin{equation}
\label{50.2} \beta^\pm_n (z) =\beta^\pm_n (z_n^*) (1+ O(\eta_n) )
\quad \text{if}  \quad |z-z_n^*| \leq 2(|\beta^-_n (z_n^*)|
+|\beta^+_n (z_n^*)|),
\end{equation}
where $z_n^* = \frac{1}{2}(z_n^+ + z_n^-)$
with $z_n^\pm = \lambda_n^\pm - \lambda_n^0,\;$
$\lambda_n^0 = n^2 $ for Hill operators and $\lambda_n^0 = n $
for Dirac operators, and
$(\eta_n)$ is a sequence of positive numbers such that $\eta_n
\to 0.$

There are wide classes of potentials that satisfy (\ref{50.1}) and (\ref{50.2}).
We give examples of such classes in Sections 5--9 below.

Set
\begin{equation}
\label{521}
\rho_n = 2(|\beta_n^- (z_n^*)|+|\beta_n^+ (z_n^*)|), \quad
D_n = \{z: |z-z_n^*|  \leq \rho_n\};
\end{equation}
then, for Hill or Dirac operators with $L^2 $ potentials,
\begin{equation}
\label{521a}
\rho_n  \to 0  \quad \text{as} \quad |n| \to \infty
\end{equation}
(see (2.51) in \cite[Proposition 28]{DM15}
or (2.109) in \cite[Proposition 35]{DM15}.
For Hill operators with $H^{-1}_{per}$-potentials we have
\begin{equation}
\label{521b}
\rho_n/n  \to 0  \quad \text{as} \quad n \to \infty
\end{equation}
( see (2.4) and (4.32) in \cite{DM21}).
In view of (\ref{521a}) and (\ref{521b})
it follows that
\begin{equation}
\label{522} \varepsilon_n: = \sup_{D_n} \left
|\frac{\partial \alpha_n}{\partial z } \right|+ \sup_{D_n} \left |\frac{\partial
\beta_n^-}{\partial z }\right| + \sup_{D_n} \left
|\frac{\partial \beta_n^+}{\partial z }\right| \to 0
\end{equation}
(see \cite[Lemma 29]{DM15} for
Hill operators with $L^2$-potentials or  (4.33) in \cite[Proposition~15]{DM21} for
Hill operators with singular potentials, and (2.109) in
\cite[Proposition~35]{DM15} for Dirac operators with $L^2$-potentials).

\begin{Remark}
\label{R1} Under the above assumptions, we have $\gamma_n \neq 0 $
for all but finitely many $n\in \Delta. $
\end{Remark}

\begin{proof}
By  Lemma~\ref{basic},
 $\; \lambda_n^-=\lambda_n^+ $ if and only if
the basic equation (\ref{p2}) has a double root $z= z_n^*, $ i.e., if
and only if $z_n^*$ satisfies simultaneously (\ref{p2}) and
\begin{equation}
\label{p3} 2(z- \alpha_n (z)) \left (1- \frac{d\alpha_n}{dz}  (z)
\right ) = \frac{d}{dz}(\beta_n^-\beta_n^+) (z).
\end{equation}
In view of Cauchy formula for the derivative, from (\ref{50.2}) and (\ref{521}) it
follows that
$$
\left |\frac{d}{dz}(\beta_n^-\beta_n^+) (z_n^*) \right | \leq
\frac{1}{\rho_n}|\beta_n^- (z_n^*)\beta_n^+ (z_n^*)| (1+ C \eta_n)
$$
where  $C >0$ is
an absolute constant. On the other hand, (\ref{p2}) implies that
$$ |z_n^* - \alpha_n (z_n^*) | =|\beta_n^- (z_n^*)\beta_n^+
(z_n^*)|^{1/2}, $$ so by (\ref{p3}) we infer that
$$
1- \left |\frac{d\alpha_n}{dz}  (z_n^*)\right | \leq \frac{1}{2}
\frac{|\beta_n^- (z_n^*)\beta_n^+ (z_n^*)|^{1/2}}{|\beta_n^-
(z_n^*)|+|\beta_n^+ (z_n^*)|} (1+ C \eta_n) \leq \frac{1}{4} (1+ C
\eta_n).
$$
Since $\frac{d\alpha_n}{dz}  (z_n^*) \to 0 $ and $\eta_n \to 0,$ the
latter inequality may hold for at most finitely many $n \in \Delta. $
\end{proof}

\begin{Remark}
\label{R2} If (\ref{50.1}) holds and there is a constant $c\geq 1 $
such that
\begin{equation}
\label{50.2a} \frac{1}{c}|\beta_n^+ (z_n^*)| \leq |\beta_n^- (z_n^*)|
\leq c|\beta_n^+ (z_n^*)| \quad  \text{for} \; \; n\in \Delta,
\end{equation}
then (\ref{50.2}) holds as well.
\end{Remark}

\begin{proof}
In view of (\ref{521}), (\ref{522}) and (\ref{50.1}),  if $z \in D_n$ then we have
$$
\begin{aligned}
|\beta_n^+ (z)-\beta_n^+ (z_n^*)|  &\leq \varepsilon_n
|z-z_n^*|\\
&\leq 2\varepsilon_n (|\beta_n^- (z_n^*)|+|\beta_n^+ (z_n^*)|) \\
&\leq 2\varepsilon_n (c+1)|\beta_n^+ (z_n^*)|,
\end{aligned}
$$
so $\beta_n^+  $ satisfies (\ref{50.2}) with $\eta_n = 2  (c+1)
\varepsilon_n.$   The proof is the same  for $\beta_n^-.$
\end{proof}

\begin{Theorem}
\label{thm50} Let $L (v)$ be either the Hill operator
with a potential $v \in H_{per}^{-1}([0,\pi])$
or the
one-dimensional Dirac operator with a
potential  $v= \begin{pmatrix} 0 & P
\\ Q & 0
\end{pmatrix},$  where $P,Q \in L^2 ([0,\pi]).$
Suppose that there is an infinite set $\Delta $ of even (if $bc
=Per^+$) or odd (if $bc =Per^-$) integers such that
(\ref{50.1}) and (\ref{50.2}) hold for $n \in
\Delta. $  Then, for  $n \in \Delta, $
there are analytic  branches $\sqrt{\beta_n^\pm (z)}$ of
$[\beta_n^\pm (z)]^{1/2}$
(defined in a neighborhood of $D_n \in (\ref{521})$)
such that  $ z^+_n  $ and $z_n^- $ are, respectively, the only roots of the equations
\begin{equation}
\label{51.63}  z - \alpha_n (z)= \sqrt{\beta_n^-
(z)}\sqrt{ \beta_n^+ (z)},   \quad z\in D_n,
\end{equation}
\begin{equation}
\label{51.64}    z - \alpha_n (z)= - \sqrt{\beta_n^-
(z)}\sqrt{ \beta_n^+ (z)},   \quad z\in D_n.
\end{equation}
Moreover, we have
\begin{equation}
\label{50.3} \gamma_n \sim 2\sqrt{\beta_n^- (z_n^*)} \sqrt{\beta_n^+
(z_n^*)},  \quad n\in \Delta.
\end{equation}

\end{Theorem}

\begin{proof}
We consider only indices $n \in \Delta. $  By (\ref{50.1}) and
(\ref{50.2}), if $n \in \Delta $ and $|n|$ is large enough, then
$$\beta_n^\pm (z) =\beta^\pm_n (z_n^*) (1+ O(\eta_n) ) \quad \text{for} \quad
z \in  D_n. $$
We set
\begin{equation}
\label{51.630} 
\sqrt{\beta_n^\pm (z)} =   \sqrt{|\beta_n^\pm (z^*_n)|} \,
e^{i\frac{\theta_n^\pm}{2}} \left (\frac{\beta_n^\pm
(z)}{\beta_n^\pm (z^*_n)}   \right )^{1/2},
\end{equation}
where $\theta_n^\pm = arg \,\beta_n^\pm (z^*_n) \in (-\pi, \pi]$  and
the third factor is defined by  the
Taylor series of $\zeta^{1/2}$
about $\zeta =1.$ Then
$\sqrt{\beta_n^- (z)}$ and $ \sqrt{\beta_n^+ (z)}$ are well-defined
analytic functions in a neighborhood of the closed disc
$D_n, $  so the basic equation (\ref{p2}), that is
$$
 (z-\alpha_n (v;z))^2=  \beta^-_n (v;z)\beta^+_n (v;z),
$$
splits into the equations  (\ref{51.63}) and (\ref{51.64}).

Next we show, for large enough $|n|,$  that each of the equations
(\ref{51.63}) and (\ref{51.64}) has exactly one root in the disc $
D_n.$ Consider the function
\begin{equation}
\label{50.620}
\varphi_n (z) = z-\alpha_n (z) -  \sqrt{\beta_n^- (z)}\sqrt{ \beta_n^+ (z)},
\quad z\in D_n.
\end{equation}
By (\ref{50.2}) and (\ref{521}), we have
$$
\sup \left \{\left |\sqrt{\beta_n^- (z) }\sqrt{\beta_n^+ (z)}\right |: \;
z\in D_n \right \} \leq \sqrt{|\beta_n^- (z_n^*)|}\sqrt{|\beta_n^+ (z_n^*)|}
\; (1+ C \eta_n),
$$
where $C>0$ is an absolute constant. Therefore, from the Cauchy
formula for the derivative it follows, for  $ z\in D^*_n:= \{z: \, |z-z_n^*| < \rho_n/2\}, $ that
\begin{align}
\label{50.630}
 \left |\frac{d}{dz}\left (\sqrt{\beta_n^- (z)}\sqrt{ \beta_n^+ (z)}\right )\right | & \leq
 \frac{2}{\rho_n}\sqrt{|\beta_n^- (z_n^*)|}\sqrt{ |\beta_n^+ (z_n^*)|}
\, (1+ C \eta_n)\\  \nonumber
&= \frac{\sqrt{|\beta_n^- (z_n^*)|}\sqrt{ |\beta_n^+ (z_n^*)|}}{|\beta_n^-
(z_n^*)|+ | \beta_n^+ (z_n^*)|} \, (1+ C \eta_n)\\  \nonumber
& \leq \frac{1}{2}(1+ C \eta_n).
\end{align}

By (\ref{50.620}), (\ref{522})  and (\ref{50.630}),  for large
enough $|n| $ we have
$$
|1-\varphi_n^\prime (z)| \leq \left  |\frac{d\alpha_n}{dz}(z) \right |
+ \left |\frac{d}{dz}\left (\sqrt{\beta_n^- (z)}\sqrt{ \beta_n^+ (z)}\right )\right | \leq
\frac{2}{3} \quad \text{for} \;\; z\in D^*_n.
$$
Therefore, if $z_1, z_2 \in D^*_n$ then
$$
|(z_2 -\varphi_n (z_2)) - (z_1 -\varphi_n (z_1))| =\left |
\int_{[z_1,z_2]} (1-\varphi_n^\prime (\zeta))d\zeta \right| \leq
\frac{2}{3} |z_2 -z_1|.
$$
Now the triangle inequality implies that
\begin{equation*}
\label{50.22} |\varphi_n (z_2) - \varphi_n (z_1) |  \geq \frac{1}{3} |z_2
- z_1|, \quad z_1, z_2 \in D^*_n.
\end{equation*}
Hence the equation (\ref{51.63}) has at most one root in
$D^*_n. $ Of course, the same argument shows that
the equation (\ref{51.64}) has also at most one root in
$D^*_n. $

Moreover, by (\ref{32.4a}) we
have $|\gamma_n| \leq \rho_n, $ and since $|z_n^\pm - z_n^*|=
|\gamma_n|/2$ it follows  that $z_n^-, z_n^+ \in
D^*_n.$
On the other hand,  in view of (\ref{521a}) and (\ref{521b})  Lemma~\ref{basic}
implies that $z_n^+ = \lambda_n^+ -n^2$ and $z_n^- =
\lambda_n^- - n^2 $ are the only roots of
the equation (\ref{p2}) in $D_n. $
Thus, we infer that for large enough
$|n|$  each of the equations (\ref{51.63}) and (\ref{51.64}) has
exactly one root in $D_n.$

Multiplying by $-1,$ if necessary, the expression in (\ref{51.630})  
which defines $\sqrt{\beta_n^- (z)}, $  we can always achieve 
that $z_n^+$ is the only
root of (\ref{51.63}) in $D_n $ and $z_n^-$ is the only root of
(\ref{51.64}) in  $D_n.$    Set
\begin{equation}
\label{50.63}  \zeta_n^+ :=z_n^+ - \alpha_n (z_n^+)= \sqrt{\beta_n^-
(z_n^+)}\sqrt{ \beta_n^+ (z_n^+)},
\end{equation}
\begin{equation}
\label{50.64}  \zeta_n^- := z_n^- -\alpha_n (z_n^-)=  -
\sqrt{\beta_n^- (z_n^-)}\sqrt{ \beta_n^+ (z_n^-)}.
\end{equation}
From (\ref{50.2}) it follows that
\begin{equation}
\label{50.15} \zeta_n^+ - \zeta_n^- =
2\sqrt{\beta_n^- (z_n^*)}\sqrt{\beta_n^+ (z_n^*)} \,(1+O(\eta_n)).
\end{equation}
On the other hand, by (\ref{522}) we know that
$$ |\alpha(z_n^+) - \alpha(z_n^-)| \leq \varepsilon_n
|z_n^+ -z_n^-|,
$$
with $\varepsilon_n \to 0.$ Therefore, by (\ref{50.63}) and
(\ref{50.64}) we infer
\begin{equation}
\label{50.150} \frac{\zeta^+_n - \zeta_n^-}{z_n^+ -z_n^-}= 1 -
\frac{\alpha(z_n^+) - \alpha(z_n^-)}{z_n^+ -z_n^-} = 1+
O(\varepsilon_n).
\end{equation}
Obviously, (\ref{50.15}) and (\ref{50.150}) imply (\ref{50.3}).

\end{proof}

\section{Asymptotics of deviations $\mu_n- \lambda_n^+$}

Under the assumptions (\ref{50.1}) and (\ref{50.2}), if $n \in
\Delta$ and  $n$ (or $|n|$) is sufficiently large, then the
Hill--Schr\"odinger (or Dirac) operator  has close to
$\lambda^0_n= n^2$ (respectively, $\lambda^0_n= n$)  one simple
Dirichlet eigenvalue $\mu_n $ and two (periodic for even $n,$ and
anti--periodic for odd $n$) simple eigenvalues $\lambda_n^-$ and $
\lambda_n^+ $ (see Remark~\ref{R1}).
In the sequel we fix an  $n\in \Delta $ and suppress the dependence
on $n$ in the notations (i.e., we write $\mu, \lambda^-, \lambda^+ $
instead of $\mu_n, \lambda_n^-, \lambda_n^+ $ etc). Moreover,
below $L_{Dir}$ denotes the Hill or Dirac operator with Dirichlet boundary conditions
but we write only  $L$ for the Hill or Dirac operator with periodic or antiperiodic
 boundary conditions.

Let $g, f, h $ be unit eigenfunctions corresponding to $\mu,
\lambda^+, \lambda^-, $ i.e.,
\begin{equation}
\label{41.0} L_{Dir} g = \mu g, \quad L f = \lambda^+ f,
\quad Lh = \lambda^- h.
\end{equation}
We denote by $$P_{Dir} =\frac{1}{2\pi i} \int_{|z-n^2|=r_n}
(z-L_{Dir})^{-1},    \quad P =\frac{1}{2\pi i} \int_{|z-n^2|=r_n}
(z-L)^{-1} $$   the corresponding Cauchy--Riesz projections of
$L_{Dir}$ and $L,$ and denote by $S$  and $E$ the ranges of
$P_{Dir} $ and $P.$ Of course, $S$ is the one-dimensional subspace
generated by  $g,$ and $E$ is the two-dimensional subspace generated
by $f$ and $ h.$

Let $P_{Dir}^0$  and $P^0$  be the Cauchy--Riesz projections of the
free operator $L^0$ corresponding to the eigenvalue $\lambda^0,$
where $\lambda^0 = n^2$ in the Hill case and $\lambda^0 = n$ in the
Dirac case.
 It is well known (see Proposition~44 or Theorem~45 in
 \cite{DM21} for Hill operators with $H^{-1}_{per}$ potentials,
 and Proposition~19 in \cite{DM15} for Dirac operators),
 that there is a sequence of
 positive numbers $\kappa_n \to 0$ such that
\begin{equation}
\label{41.1} \|P - P^0 : \; L^2 \to L^\infty \| \leq \kappa_n, \quad
\|P_{Dir} - P_{Dir}^0 : \; L^2 \to L^\infty \|  \leq \kappa_n.
\end{equation}

The subspace $E^0 = Ran \,P^0$ has the following standard basis of
eigenvectors of $L^0 $ (corresponding to the eigenvalue $\lambda^0 =
\lambda^0_n $ of $L^0,$ where $\lambda^0_n = n^2 $ or  $n,$
respectively):
\begin{equation}
\label{41.2} e^1 (x)= e^{-inx}, \quad e^2 (x)= e^{inx}, \quad n \in
\mathbb{N},
\end{equation}
in the Hill--Schr\"odinger case, and
\begin{equation}
\label{41.3} e^1 (x)=  \begin{pmatrix} e^{-inx} \\0
\end{pmatrix}, \quad e^2 (x)= \begin{pmatrix}0\\ e^{inx} \end{pmatrix},
\quad n\in \mathbb{Z},
\end{equation}
in the Dirac case. The subspace $S^0 = Ran \,P^0_{Dir} $ is generated
by \begin{equation}
\label{41.5}
g^0 = \begin{cases}
\frac{1}{\sqrt{2}} (e^2_n - e^1_n) & \text{for Hill operators},\\
\frac{1}{\sqrt{2}} (e^1_n + e^2_n)  & \text{for Dirac operators}.
\end{cases}
\end{equation}

In order to consider  the Hill--Schr\"odinger and Dirac cases
simultaneously we set
\begin{equation}
\label{42.01} \ell_0 (y) = \begin{cases} y(0) & \text{in the
Hill--Schr\"odinger case,}\\ y_2 (0) - y_1 (0) & \text{in the Dirac
case},
\end{cases}
\end{equation}
\begin{equation}
\label{42.02} \ell_1 (h) = \begin{cases} y(\pi) & \text{in the
Hill--Schr\"odinger case,}\\ y_2 (\pi) - y_1 (\pi) & \text{in the
Dirac case},
\end{cases}
\end{equation}
where $y(x) $ is a complex--valued function in the
Hill--Schr\"odinger case, and $y(x) = \begin{pmatrix} y_1(x)\\ y_2
(x) \end{pmatrix}$ in the Dirac case. Then $y$  satisfies the
Dirichlet boundary conditions if and only if $$ \ell_0 (y) = 0, \quad
\ell_1 (y) = 0. $$

\begin{Theorem}
\label{thm51} Assume that (\ref{50.1}) and (\ref{50.2}) hold, and the
branches  $\sqrt{\beta_n^- (z)}$  and
$\sqrt{\beta_n^+ (z)}$ are chosen as in Theorem~\ref{thm50} so that
the second equalities in (\ref{50.63}) and (\ref{50.64}) hold.

(a) If $\;-1$ is not a cluster point of the sequence $\left (
\sqrt{\beta^-_n (z^*_n)} / \sqrt{\beta^+_n (z^*_n)}\right)_{n\in
\Delta},$ then in the Hill case
\begin{equation}
\label{50.4} \mu_n - \lambda_n^+ \sim  -\frac{1}{2} \left
(\sqrt{\beta^+_n (z_n^*)} + \sqrt{\beta^-_n (z_n^*)} \right )^2,
\quad n \in \Delta,
\end{equation}
and in the Dirac case
\begin{equation}
\label{50.5} \mu_n - \lambda_n^- \sim  \frac{1}{2} \left
(\sqrt{\beta^+_n (z_n^*)} + \sqrt{\beta^-_n (z_n^*)} \right )^2,
\quad  n \in \Delta.
\end{equation}

(b) If $\,1$ is not a cluster point of the sequence $\left (
\sqrt{\beta^-_n (z^*_n)}/\sqrt{\beta^+_n (z^*_n)}\right)_{n\in
\Delta},$ then in the Hill case
\begin{equation}
\label{50.8} \mu_n - \lambda_n^- \sim  -\frac{1}{2} \left
(\sqrt{\beta^+_n (z_n^*)} - \sqrt{\beta^-_n (z_n^*)} \right )^2,
\quad  n \in \Delta.
\end{equation}
and in the Dirac case
\begin{equation}
\label{50.9} \mu_n - \lambda_n^+ \sim  \frac{1}{2} \left
(\sqrt{\beta^+_n (z_n^*)} -\sqrt{\beta^-_n (z_n^*)} \right )^2,
\quad n \in \Delta.
\end{equation}

(c)  If $-1$ is not a cluster point of the sequence $\left (\beta^-_n
(z^*_n)/\beta^+_n (z^*_n)\right )_{n \in \Delta},$ then in the Hill
case
\begin{equation}
\label{500.4} \mu_n - \frac{1}{2} \left (\lambda_n^- + \lambda_n^+
\right ) \sim -\frac{1}{2} \left ( \beta^+_n (z_n^*) + \beta^-_n
(z_n^*)\right ), \quad n \in \Delta,
\end{equation}
and in the Dirac case
\begin{equation}
\label{500.5} \mu_n - \frac{1}{2} \left ( \lambda_n^- + \lambda_n^+
\right ) \sim \frac{1}{2} \left ( \beta^+_n (z_n^*) + \beta^-_n
(z_n^*) \right ), \quad n \in \Delta.
\end{equation}

\end{Theorem}

\begin{proof}
The method of our proof comes from \cite{DM5}; see also \cite{DM15, DM21}
where the same method was developed for  Dirac operators and 
Hill operators  with singular potentials.

The proof  consists of several steps.
\bigskip

1. Fix $n \in \Delta$ and let $f, h, g $ be unit eigenfunctions as
in (\ref{41.0}). Fix a vector $\varphi = af+ bg \in E $ so that
\begin{equation}
\label{51.00}   \varphi  \perp f, \quad \|\varphi \|=1.
\end{equation}
Then we have $ L\varphi = a \lambda^+ f + b \lambda^- h = \lambda^+
\varphi - \gamma b h, $  that is
\begin{equation}
\label{51.01}   L\varphi =\lambda^+ \varphi - \xi h,
\end{equation}
where $\xi = \gamma b$ is a constant.

Consider the function
\begin{equation}
\label{51.02} G(x) = \ell_0 (\varphi) f(x) - \ell_0 (f) \varphi (x).
\end{equation}
Since $f$ and $\varphi $ satisfy periodic or antiperiodic boundary
conditions, it follows that $G$ satisfies $\ell_0 (G) = \ell_1 (G)=0.$
Hence, $G$ is in the domain of the operator $L_{Dir}.$

In view of (\ref{51.01}), we have
$$
LG =\ell_0 (\varphi) \lambda^+ f - \ell_0 (f) \lambda^+ \varphi +
\ell_0 (f) \xi h = \lambda^+ G +\ell_0 (f) \xi h.
$$

Let $\tilde{g}$ be a unit eigenvector of the adjoint operator
$(L_{Dir} (v))^* $ corresponding to the eigenvalue $\overline{\mu} =
\overline{\mu}_n.$ Then on one hand we have
$$
\langle LG, \tilde{g} \rangle =  \lambda^+ \langle G, \tilde{g}
\rangle + \xi \ell_0 (f) \langle h, \tilde{g} \rangle,
$$
and on the other hand
$$
\langle LG, \tilde{g} \rangle = \langle G, (L_{Dir})^* \tilde{g}
\rangle= \mu \langle G, \tilde{g} \rangle.
$$
Therefore, we obtain that
\begin{equation}
\label{51.10} (\mu - \lambda^+)\langle G, \tilde{g} \rangle =\xi
\ell_0 (f) \langle h, \tilde{g} \rangle.
\end{equation}
\bigskip

2. We use (\ref{51.10}) to find the asymptotics of $\mu - \lambda^+.
$ To this end we estimate the asymptotics of $\langle G, \tilde{g}
\rangle, \; \ell_0 (f)$  and     $\langle h, \tilde{g} \rangle $
after "replacing" $f, \, h, \, \varphi, \,\tilde{g} $ respectively
by
\begin{equation}
\label{51.12} f^0 = \frac{P^0 f}{\|P^0f\|_2}, \quad h^0 = \frac{P^0
h}{\|P^0h\|_2}, \quad \varphi^0 = \frac{P^0 \varphi}{\|P^0
\varphi\|_2}, \quad  \tilde{g}^0 = \frac{P_{Dir}^0
\tilde{g}}{\|P_{Dir}^0 \tilde{g}\|_2}.
\end{equation}
Since $S^0 $ is one-dimensional space, we have $\tilde{g}^0 =
e^{i\theta}g^0.$ Therefore, we may assume without loss of generality
that
$$\tilde{g}^0 = g^0 $$
(otherwise we can replace $\tilde{g}$ by  $e^{-i\theta}\tilde{g}$).

By (\ref{41.1}) it follows that there is an absolute constant $C>1$
such that
\begin{eqnarray}
\label{51.13a} \|f-f^0\|_2 \leq C\kappa_n, \quad \|f-f^0\|_\infty
\leq C \kappa_n, \\
\label{51.13b} \|h-h^0\|_2 \leq C\kappa_n, \quad \|h-h^0\|_\infty
\leq C \kappa_n,\\
\label{51.13c} \|\varphi-\varphi^0\|_2 \leq C\kappa_n, \quad
\|\varphi-\varphi^0\|_\infty
\leq C \kappa_n,\\
\label{51.13d} \|\tilde{g} - g^0\|_2 \leq C\kappa_n, \quad
\|\tilde{g}-g^0\|_\infty \leq C \kappa_n.
\end{eqnarray}
Next we prove (\ref{51.13a}) only because the other inequalities
follow from the same argument. Let us only mention that
$$
(L_{Dir} (v))^* = L_{Dir} (v^*),
$$
where $v^*(x)=\overline{v}(x) $ in the Hill case, and $v^*=
\begin{pmatrix}
0 & \overline{\mathcal{Q}} \\ \overline{\mathcal{P}}  & 0
\end{pmatrix}
$ if $v=
\begin{pmatrix}
0 & \mathcal{P} \\ \mathcal{Q}  & 0
\end{pmatrix}
$ in the Dirac case, so we may use (\ref{41.1}) for the Cauchy-Riesz
projections associated with the adjoint operator $(L_{Dir} (v))^*$
as well.

Obviously, if $y(x), \, x \in [0,\pi],$ is an  $L^\infty$ function (in
the Hill case), or an $L^\infty$  vector-function in the Dirac case,
then
 $$
\|y\|_2 \leq C_1 \|y\|_\infty
$$
with some absolute constant $C_1.$ By (\ref{41.1}),
$$
\|f-P^0 f\|_\infty = \|(P-P^0) f\|_\infty \leq \|P-P^0:\, L^2 \to
L^\infty\| \leq \kappa_n,
$$
so  $\|f-P^0 f\|_2 \leq C_1 \kappa_n,$  and  it follows that
$$
\|P^0f\|_2 \geq \|f\|_2- \|f-P^0f\|_2 \geq 1- C_1 \kappa_n.
$$
Therefore, in view of (\ref{51.12}), we obtain
\begin{align*}
\|f-f^0\|_\infty &\leq \|f-P^0 f\|_\infty + \|P^0 f - f^0\|_\infty\\
 &\leq \kappa_n +  (1-\|P^0f\|_2) \|f^0\|_\infty \\ & \leq
(1+ C_1 \|f^0\|_\infty) \kappa_n.
\end{align*}
Since $f^0$ is a unit vector in $E^0,$ it has the form $f^0 = f^0_1
e^1_n + f^0_2 e^2_n,   $  so the Cauchy inequality implies that $
|f^0_1| + |f^0_2| \leq \sqrt{2}.$ In view of (\ref{41.2}), we have
$\|e^1_n\|_\infty =\|e^2_n\|_\infty =1,$ and therefore
$$
\|f^0\|_\infty \leq |f^0_1| + |f^0_2| \leq \sqrt{2}.
$$
Thus, $ \|f-f^0\|_\infty \leq  (1+C_1 \sqrt{2}) \kappa_n,$ and $
\|f-f^0\|_2 \leq  C_1 (1+C_1 \sqrt{2}) \kappa_n,$ i.e.,
(\ref{51.13a}) holds.

One can easily see by (\ref{51.13a})--(\ref{51.13d}) that
\begin{eqnarray}
\label{51.14a} \ell_0 (f) = \ell_0 (f^0) + O(\kappa_n), \quad
\langle
f,\tilde{g} \rangle =\langle f^0, g^0 \rangle + O(\kappa_n),\\
\label{51.14b} \ell_0 (h) = \ell_0 (h^0) + O(\kappa_n), \quad
\langle h,\tilde{g} \rangle =\langle h^0, g^0 \rangle + O(\kappa_n),\\
\label{51.14c} \ell_0 (\varphi) = \ell_0 (\varphi^0) + O(\kappa_n),
\quad \langle \varphi,\tilde{g} \rangle =\langle \varphi^0, g^0
\rangle + O(\kappa_n).
\end{eqnarray}
\bigskip

3. By  Lemma \ref{basic}  we know that the vector
$$f^0=f^0_1 e^1 + f^0_2 e^2$$ given by (\ref{51.12}) is an
eigenvector of the matrix $\begin{pmatrix}
\alpha (z^+)  &  \beta^- (z^+)\\
\beta^+ (z^+) &  \alpha (z^+) \end{pmatrix}$ corresponding to its
eigenvalue $z^+,$ i.e.,
$$
\begin{pmatrix}
z^+ -\alpha (z^+)  &  -\beta^- (z^+)\\
-\beta^+ (z^+) & z^+ - \alpha (z^+)
\end{pmatrix}
\begin{pmatrix} f_1^0\\f_2^0
\end{pmatrix} = 0.
$$
In view of (\ref{50.63}), $z^+ - \alpha (z^+) = \sqrt{\beta^-
(z^+)}\sqrt{\beta^+ (z^+)},$ so it follows that
$$
f^0_1 \sqrt{\beta^+ (z^+)} - f_2^0 \sqrt{\beta^- (z^+)} = 0.
$$
From here we obtain (multiplying, if necessary, $f$ by an appropriate
factor of the form $e^{\theta}$) that
\begin{equation}
\label{51.21} f^0 = \begin{pmatrix} f^0_1\\f^0_2 \end{pmatrix} =
\frac{1}{\sqrt{\rho (z^+)}}
\begin{pmatrix} \sqrt{\beta^- (z^+)}\\{} \\\sqrt{\beta^+ (z^+)}
\end{pmatrix},
\end{equation}
where
\begin{equation}
\label{51.22} \rho (z) = |\beta^- (z)|+ |\beta^+ (z)|.
\end{equation}
By the same argument it follows that
\begin{equation}
\label{51.23} h^0 = \begin{pmatrix} h^0_1\\h^0_2 \end{pmatrix} =
\frac{1}{\sqrt{\rho (z^-)}} \begin{pmatrix} -\sqrt{\beta^-
(z^-)}\\{}\\ \sqrt{\beta^+ (z^-)}
\end{pmatrix}.
\end{equation}

Next we use that $\varphi \perp f$  to find asymptotically
the vector $\varphi^0 =\varphi^0_1 e^1 + \varphi^0_2 e^2 $
defined by (\ref{51.12}). Since
$\varphi^0 \in E^0, $  we have that $\varphi^0= c_1 f^0 + c_2
(f^0)^{\perp},$  where $(f^0)^{\perp} = \overline{f^0_2}e^1
-\overline{f^0_1}e^2. $  In view of (\ref{51.00}), (\ref{51.13a})
and (\ref{51.13c}) it follows that
$$
c_1 = \langle \varphi^0, f^0 \rangle = \langle \varphi^0-\varphi, f
\rangle + \langle \varphi^0, f^0 - f \rangle = O (\kappa_n).
$$
Therefore, $$  |c_2| = \sqrt{1-|c_1|^2} = 1 + O(\kappa_n).
$$
We may assume without loss of generality that $c_2 >0$ (otherwise we
can multiply $\varphi $ by $e^{-i\theta}$ with $\theta = \arg c_2).$
Therefore, we infer that
\begin{equation}
\label{51.25} \varphi^0 =\varphi^0_1 e^1 + \varphi^0_2 e^2, \quad
\varphi^0_1=\overline{f^0_2} + O(\kappa), \quad
\varphi^0_2=-\overline{f^0_1}+ O(\kappa).
\end{equation}
\bigskip

4. Now from (\ref{41.5}), (\ref{42.01}),
(\ref{51.13a})--(\ref{51.13d}), (\ref{51.14a}), (\ref{51.14c}) and
(\ref{51.25}) it follows that
\begin{equation}
\label{51.26} \ell_0 (f) = \begin{cases} f_2^0 +f_1^0 +
O(\kappa_n)   &  \text{in the Hill case},\\
f_2^0 - f_1^0 + O(\kappa_n)   &  \text{in the Dirac case},
\end{cases}
\end{equation}
\begin{equation}
\label{51.27} \langle f, \tilde{g} \rangle = \begin{cases}
\frac{1}{\sqrt{2}} (f_2^0 -f_1^0) +
O(\kappa_n)   &  \text{in the Hill case},\\
 \frac{1}{\sqrt{2}} (f_2^0 +f_1^0) +
O(\kappa_n)   &  \text{in the Dirac case},
\end{cases}
\end{equation}
\begin{equation}
\label{51.270} \langle h, \tilde{g} \rangle = \begin{cases}
\frac{1}{\sqrt{2}} (h_2^0 -h_1^0) +
O(\kappa_n)   &  \text{in the Hill case},\\
 \frac{1}{\sqrt{2}} (h_2^0 +h_1^0) +
O(\kappa_n)   &  \text{in the Dirac case},
\end{cases}
\end{equation}
\begin{equation}
\label{51.28} \ell_0 (\varphi) = \begin{cases} \overline{f_2^0} -
\overline{f_1^0} +
O(\kappa_n)   &  \text{in the Hill case},\\
 -(\overline{f_2^0} +\overline{f_1^0}) +
O(\kappa_n)   &  \text{in the Dirac case},
\end{cases}
\end{equation}
and
\begin{equation}
\label{51.29} \langle \varphi, \tilde{g} \rangle = \begin{cases}
-\frac{1}{\sqrt{2}} (\overline{f_2^0} +\overline{f_1^0}) +
O(\kappa_n)   &  \text{in the Hill case},\\
 \frac{1}{\sqrt{2}} (\overline{f_2^0} -\overline{f_1^0}) +
O(\kappa_n)   &  \text{in the Dirac case}.
\end{cases}
\end{equation}
In view of the above formulas we obtain
$$
\begin{aligned}
\langle G, \tilde{g} \rangle &=
  \ell_0 (\varphi)
\langle f, \tilde{g}
\rangle - \ell_0 (f) \langle \varphi, \tilde{g} \rangle\\
&=\begin{cases} (\overline{f_2^0}
-\overline{f_1^0})\frac{1}{\sqrt{2}} (f_2^0 -f_1^0)+ (f_2^0 +f_1^0)
\frac{1}{\sqrt{2}}
(\overline{f_2^0} +\overline{f_1^0}) + O(\kappa_n),\\
-(\overline{f_2^0} +\overline{f_1^0})\frac{1}{\sqrt{2}} (f_2^0
+f_1^0)- (f_2^0 - f_1^0) \frac{1}{\sqrt{2}} (\overline{f_2^0}
-\overline{f_1^0}) + O(\kappa_n),
\end{cases}
\end{aligned}
$$
where the first and second lines present the Hill and Dirac cases respectively.
Since $\|f^0\|^2 =|f_1^0|^2 +
|f_2^0|^2=1,$ one can easily see that
\begin{equation}
\label{51.30} \langle G, \tilde{g} \rangle = \begin{cases} \sqrt{2}
+O(\kappa_n) &  \text{in the Hill case},\\
-\sqrt{2} +O(\kappa_n) &  \text{in the Dirac case}.
\end{cases}
\end{equation}
\bigskip

5. By (\ref{51.01}) we have that
\begin{equation}
\label{51.31} (\lambda^+ - L) \varphi = \xi h.
\end{equation}
In the proof of Lemma~21 in \cite{DM15} it is shown that equation
\cite[(2.6)]{DM15} leads to \cite[(2.13)]{DM15}. In our notations,
this means that (\ref{51.31}) implies
\begin{equation}
\label{51.32} (\lambda - L^0 - S(\lambda)) \varphi^0 =\xi [h^0 +
P^0(1 - T)^{-1}T (h-h^0)],
\end{equation}
where the operator $ S: E^0 \to E^0 $  acts in the two-dimensional
space $E^0$ and $T=T(n)$  is a continuous operator with
$\|T\|<\frac{1}{2}$ for large enough $|n|.$  Moreover, by Section~2.2
and Section~2.4 in \cite{DM15} (see also Lemma~6 in \cite{DM21} for
the modifications needed to handle the case of singular potentials $v
\in H^{-1}_{per}$) the matrix representation of the
operator $S(\lambda) $ is $\begin{pmatrix} \alpha (z)  &  \beta^-(z) \\
\beta^+ (z) & \alpha (z)\end{pmatrix} $ with $z=\lambda - \lambda^0.$

In view of (\ref{51.13b}) we have $\|h-h^0\|= O(\kappa_n),$ so it
follows that $$ \|P^0(1 - T)^{-1}T (h-h^0)\| =O(\kappa_n).$$
Therefore, (\ref{50.63}) and (\ref{51.32}) imply that
\begin{equation}
\label{51.33}
\begin{pmatrix} \sqrt{\beta^-(z^+)}\sqrt{\beta^+(z^+)} &   -\beta^-(z^+) \vspace{1mm}
\\
-\beta^+ (z^+)  &  \sqrt{\beta^-(z^+)}\sqrt{\beta^+(z^+)} \end{pmatrix}
\begin{pmatrix}  \varphi^0_1 \vspace{1mm} \\
\varphi^0_2 \end{pmatrix} = \xi \begin{pmatrix}  h^0_1 + O(\kappa_n)
\vspace{1mm}\\  h^0_2 + O(\kappa_n) \end{pmatrix}.
\end{equation}

Assume that $|\beta^- (z^-)| \leq |\beta^+ (z^-)| $  (i.e., we
consider
 $n \in \Delta $ such that $|\beta_n^- (z_n^-)| \leq |\beta_n^+ (z_n^-)|
 $). Then by (\ref{51.23}) we have that $|h^0_2|\geq 1/\sqrt{2},$
so
$$ h^0_2 +O(\kappa_n) =h^0_2 \, (1+
O(\kappa_n)).$$ In view of (\ref{51.21})--(\ref{51.23}) and
(\ref{51.25}), the second entrees equality in the vector equation
(\ref{51.33}) implies that
\begin{equation}
\label{51.34} - \sqrt{\beta^+(z^+)} \sqrt{\rho (z^+)} (1 +  O(\kappa_n))=\xi
\frac{\sqrt{\beta^+ (z^-)}}{\sqrt{\rho (z^-)}} (1+ O(\kappa_n))
\end{equation}
By (\ref{50.2}) we have that
$$
\sqrt{\beta^+(z^\pm)} = \sqrt{\beta^+(z^*)}(1 + O(\eta_n)), \quad
\rho (z^\pm) = \rho (z^*)(1+ O(\eta_n)).
$$
Hence, by (\ref{51.34}) we infer that
\begin{equation}
\label{51.40} \xi = - \rho (z^*)(1+o(1)).
\end{equation}

In the case $|\beta^- (z^*)| \geq |\beta^+ (z^*)| $ we have
$|h^0_1|\geq 1/\sqrt{2};$ then the first entrees equality in
(\ref{51.33}) gives us that
$$
 \sqrt{\beta^-(z^+)} \sqrt{\rho (z^+)} (1 +  O(\kappa_n))= - \xi
\frac{\sqrt{\beta^- (z^-)}}{\sqrt{\rho (z^-)}} (1+ O(\kappa_n))
$$
Thus, the same argument leads to (\ref{51.40}).
\bigskip

6. By (\ref{51.26}) and (\ref{51.270}) we have that
\begin{equation}
\label{51.41} \ell_0 (f) \langle h, \tilde{g} \rangle =
\begin{cases} \frac{1}{\sqrt{2}} (f^0_2+f^0_1)(h^0_2
-h^0_1)+O(\kappa_n)
 & \text{in the Hill case},\\
\frac{1}{\sqrt{2}} (f^0_2-f^0_1)(h^0_2 + h^0_1)+ O(\kappa_n) &
\text{in the Dirac case}.
\end{cases}
\end{equation}

Taking into account (\ref{50.2}) and (\ref{51.21})--(\ref{51.23}) we
obtain that
\begin{equation}
\label{51.42} \ell_0 (f) \langle h, \tilde{g} \rangle =
\begin{cases} \frac{1}{\sqrt{2}\rho (z^*)} \left (\sqrt{\beta^+
(z^*)}
+ \sqrt{\beta^- (z^*)} \right )^2 + o(1) & \text{in the Hill case},\\
\frac{1}{\sqrt{2}\rho (z^*)} \left (\sqrt{\beta^+ (z^*)}
-\sqrt{\beta^- (z^*)} \right )^2 + o(1) &  \text{in the Dirac case}.
\end{cases}
\end{equation}

If $\,-1$ is not a cluster point of the sequence $\left (
\sqrt{\beta^-_n (z^*_n)}/\sqrt{\beta^+_n (z^*_n)}\right)_{n\in
\Delta},$ then the sequence
$$
\frac{1}{\rho (z^*)} \left (\sqrt{\beta_n^+ (z_n^*)} +
\sqrt{\beta_n^- (z_n^*)} \right )^2= \frac{\beta_n^+
(z_n^*)}{|\beta_n^+ (z_n^*)|} \frac{\left ( 1+ \sqrt{\beta_n^-
(z_n^*)}/\sqrt{\beta_n^+ (z_n^*)} \right )^2}{1+|\beta_n^-
(z_n^*)|/|\beta_n^+ (z_n^*)|}
$$
is separated from $0.$ Therefore, in the Hill case we have
\begin{equation}
\label{51.44} \ell_0 (f) \langle h, \tilde{g} \rangle =
\frac{1}{\sqrt{2}\rho (z^*)} \left (\sqrt{\beta^+ (z^*)} +
\sqrt{\beta^- (z^*)} \right )^2 (1+ o(1)).
\end{equation}
Now (\ref{51.10}), (\ref{51.30}), (\ref{51.40}) and (\ref{51.44})
imply that (\ref{50.4}) holds.

The same argument shows that if $1$ is not a cluster point of the
sequence $\left ( \sqrt{\beta^-_n (z^*_n)}/\sqrt{\beta^+_n
(z^*_n)}\right)_{n\in \Delta},$ then in the Dirac case
\begin{equation}
\label{51.45} \ell_0 (f) \langle h, \tilde{g} \rangle =
\frac{1}{\sqrt{2}\rho (z^*)} \left (\sqrt{\beta^+ (z^*)} -
\sqrt{\beta^- (z^*)} \right )^2 (1+ o(1)).
\end{equation}
This, together with (\ref{51.10}), (\ref{51.30}) and (\ref{51.40}),
implies (\ref{50.9}). \bigskip

7.  If $\,1$ and $\,-1$ are not cluster points of  $\left (
\sqrt{\beta^-_n (z^*_n)}/\sqrt{\beta^+_n (z^*_n)}\right)_{n\in
\Delta},$ then in the Hill case (\ref{50.3}) and (\ref{50.4}) imply
(\ref{50.8}). Indeed, since $\mu - \lambda^-= \mu - \lambda^+ +
\gamma$, we have
$$
\begin{aligned}
\mu_n  - \lambda_n^- & =  -\frac{1}{2} \left (\sqrt{\beta^+_n
(z_n^*)} + \sqrt{\beta^-_n (z_n^*)} \right )^2 \,\, [1+o(1)]  \\&\quad+
2\sqrt{\beta^+_n (z_n^*)}\sqrt{\beta^-_n (z_n^*)}\, [1+o(1)]\\
& =-\frac{1}{2} \left (\sqrt{\beta^+_n (z_n^*)} - \sqrt{\beta^-_n
(z_n^*)} \right )^2 (1 + A_n \, o(1)  + B_n \, o(1)),
\end{aligned}
$$
where
$$A_n =
\frac{\left (1+ \sqrt{\beta^-_n (z_n^*)}/\sqrt{\beta^+_n (z_n^*)}
\right )^2 } {\left (1- \sqrt{\beta^-_n (z_n^*)} /\sqrt{\beta^+_n
(z_n^*)} \right )^{2}}, \quad B_n = \frac{ \sqrt{\beta^-_n
(z_n^*)}/\sqrt{\beta^+_n (z_n^*)}} {\left (1- \sqrt{\beta^-_n (z_n^*)}
/\sqrt{\beta^+_n (z_n^*)} \right )^{2}}
$$
are bounded sequences. In the Dirac case, the same argument shows
that (\ref{50.3}) and (\ref{50.5}) imply (\ref{50.9}).

If  $\,-1$ is  a cluster point of  $\left ( \sqrt{\beta^-_n
(z^*_n)}/\sqrt{\beta^+_n (z^*_n)}\right)_{n\in \Delta}$ but $\,1$ is
not, then we can  prove (\ref{50.8}) and (\ref{50.9})
directly by  exchanging the roles of $f$ and $h.$  Now we fix a
vector $\varphi = af+ bh $ such that
\begin{equation}
\label{51.000}   \varphi  \perp h, \quad \|\varphi \|=1.
\end{equation}
Then
\begin{equation}
\label{51.010}   L\varphi =\lambda^- \varphi - \xi f,
\end{equation}
and it follows that
\begin{equation}
\label{51.100} (\mu -
\lambda^-)\langle G, \tilde{g} \rangle =\xi \ell_0 (h) \langle f,
\tilde{g} \rangle.
\end{equation}
Then following the steps 2 -- 6 one would obtain (\ref{50.8}) and
(\ref{50.9}). We omit the details.
\bigskip

8.  Next we prove (\ref{500.4}).  The
proof of (\ref{500.5}) (the Dirac case)
is the same.

It is enough to prove (\ref{500.4}) for every subsequence $(n_k) $ in $\Delta.$
Therefore,
we may assume that  $\left (
\sqrt{\beta^-_n (z^*_n)}/\sqrt{\beta^+_n (z^*_n)}\right)_{n\in
\Delta}$ has only one cluster point.

Suppose $\;-1$ is not a cluster point of  $\left (
\sqrt{\beta^-_n (z^*_n)}/\sqrt{\beta^+_n (z^*_n)}\right)_{n\in
\Delta}$. Then by (\ref{50.3}) and (\ref{50.4}) it follows that
$$
\mu_n - \frac{1}{2} \left (\lambda_n^- + \lambda_n^+ \right ) =
-\frac{1}{2} \left (\beta^+_n (z_n^*) + \beta^-_n (z_n^*)\right )
[1+o(1)] + B_n \, o(1),
$$
where $B_n = \sqrt{\beta^-_n (z^*_n)}\sqrt{\beta^+_n (z^*_n)}. $
Since $-1$ is not a cluster point of the sequence $(\beta^-_n
(z^*_n)/\beta^+_n (z^*_n)),$  the sequence $$
\frac{B_n}{\beta^-_n (z^*_n)+\beta^+_n (z^*_n)} =
\frac{\sqrt{\beta^-_n (z^*_n)}}{\sqrt{\beta^+_n (z^*_n)}} \left ( 1+
\frac{\beta^-_n (z^*_n)}{\beta^+_n (z^*_n)} \right )^{-1}
$$
is bounded. Hence, (\ref{500.4}) holds.

If $\;-1$ is (the only) cluster point of  $\left (
\sqrt{\beta^-_n (z^*_n)}/\sqrt{\beta^+_n (z^*_n)}\right)_{n\in
\Delta},$  then $\, 1$ is not a cluster point.
Then (\ref{500.4}) follows from (\ref{50.3}) and (\ref{50.8}).

\end{proof}

\begin{Corollary}
\label{cor14}
Suppose the assumptions of Theorem~\ref{thm51} hold and
the operator $L (v)$ is self-adjoint. Then there are branches 
$\sqrt{\beta_n^\pm (z)}$  with
\begin{equation}
\label{14.0}
\sqrt{\beta_n^- (z)} =  \overline{\sqrt{\beta_n^+ (\overline{z})}}
\end{equation}
such that Theorem~\ref{thm50} hold. Moreover,
Theorem~\ref{thm51} implies the following.

(a)  If $\; -1$ is not a cluster point of  
$\left (Re \, \beta^+_n (z^*_n)/ |\beta^+_n (z^*_n)| \right)_{n\in
\Delta},$ then in the Hill case
\begin{equation}
\label{14.1}
\mu_n - \lambda_n^+ \sim - Re\, \beta_n^+ (z_n^*)  -
|\beta_n^+ (z_n^*)|, \quad n \in \Delta,
\end{equation}
and in the Dirac case
\begin{equation}
\label{14.2}
\mu_n - \lambda_n^- \sim   Re\, \beta_n^+ (z_n^*)  + |\beta_n^+ (z_n^*)|,
\quad n \in \Delta.
\end{equation}

(b) If $\; 1$ is not a cluster point of  
$\left (Re \, \beta^+_n (z^*_n)/ |\beta^+_n (z^*_n)| \right)_{n\in
\Delta},$
 then in the Hill case
\begin{equation}
\label{14.3}
\mu_n - \lambda_n^- \sim - Re\, \beta_n^+ (z_n^*)  +
|\beta_n^+ (z_n^*)|, \quad n \in \Delta,
\end{equation}
and in the Dirac case
\begin{equation}
\label{14.4}
\mu_n - \lambda_n^+ \sim   Re\, \beta_n^+ (z_n^*)  - |\beta_n^+ (z_n^*)|,
\quad n \in \Delta.
\end{equation}

(c)  If $\;-1$ is not a cluster point of  $\left (
\beta^-_n (z^*_n)/ \beta^+_n (z^*_n)\right)_{n\in
\Delta},$  then in the Hill case
\begin{equation}
\label{14.5}
\mu_n - \frac{1}{2} (\lambda_n^- + \lambda_n^+ ) \sim -
Re\, \beta_n^+ (z_n^*), \quad n \in \Delta,
\end{equation}
and in the Dirac case
\begin{equation}
\label{14.6}
\mu_n - \frac{1}{2} (\lambda_n^- + \lambda_n^+ )  \sim
 Re\, \beta_n^+ (z_n^*), \quad n \in \Delta.
\end{equation}

\end{Corollary}

\begin{proof}
Recall that if the operator $L (v)$ is self-adjoint  (i.e.,  in the Hill case
if $v(x)$  is real valued, and in the Dirac case if
$v(x) = \begin{pmatrix}  0 & P(x) \\ Q(x) & 0 \end{pmatrix} $
with $Q(x) = \overline{P(x)}\;$), then
$$
\beta_n^- (z) = \overline{\beta_n^+ (\overline{z})}
$$
(for Hill and Dirac operators with $L^2$ potentials see, respectively:   Lemma~24,
formula (2.32) or  Lemma~30, formula (2.80)  in \cite{DM15});
for Hill operators with singular potentials
see formula (3.33) in \cite[Lemma 9]{DM21}).

Since the spectrum is real, the numbers $ z_n^*$ are also real, so
$$
\beta_n^- (z_n^*) = \overline{\beta_n^+ (z_n^*)}.
$$

If $\arg  \beta_n^+ (z_n^*) = \theta_n^+ \in (-\pi, \pi), $  then we have
$\arg  \beta_n^- (z_n^*) = \theta_n^-  = - \theta_n^+, $
so if we define the branches  $\sqrt{\beta_n^- (z)}$ and $\sqrt{\beta_n^+ (z)}$
by (\ref{51.630}) then (\ref{14.0}) holds.

However, in the case when $\theta_n^+ = \pi$  
(i.e., if $\beta_n^- (z_n^*)=\beta_n^+ (z_n^*) $ is a negative real number), 
(\ref{51.630}) does not imply (\ref{14.0}). Therefore, in this case 
we modify (\ref{51.630}) and set
$$
\sqrt{\beta_n^\pm (z)}   = \pm i \sqrt{|\beta_n^+ (z_n^*)|}   
\left (\frac{\beta_n^\pm
(z)}{\beta_n^\pm (z^*_n)}   \right )^{1/2};
$$
then (\ref{14.0}) holds as well.

In view of (\ref{14.0}), since $z_n^*$ is real we have
$$
\sqrt{\beta_n^- (z_n^*)} =\overline{\sqrt{\beta_n^+ (z_n^*)}}, \quad
\gamma_n \sim 2\sqrt{\beta_n^- (z_n^*)} \cdot \sqrt{\beta_n^+ (z_n^*)} = 2\left
|\beta_n^+ (z_n^*) \right |$$
which agree with $\gamma_n = \lambda_n^+ - \lambda_n^-  >0 $  and  shows that
this choice of branches is good for Theorem~\ref{thm50}.

It is easy to see  (by (\ref{14.0}), since   
$\arg \sqrt{\beta_n^+ (z_n^*)} $   approaches $ \pm \pi/2$
or $0$ if and only if  $ \arg  \beta_n^+ (z_n^*) $ approaches $\pm \pi $ or $0$)
that $-1 $  (or $1$)  is a cluster point of 
$\left (\sqrt{\beta^-_n (z^*_n)}/\sqrt{\beta^+_n (z^*_n)} \right)_{n\in
\Delta} $  if and only if
 $ -1$ (respectively $1$) is a cluster point of  
$\left (Re \, \beta^+_n (z^*_n)/ |\beta^+_n (z^*_n)| \right)_{n\in
\Delta}.$  
Therefore, by (\ref{14.0}) formulas (\ref{14.1})--(\ref{14.6}) follow immediately from
Theorem~\ref{thm51}.

\end{proof}

\begin{Corollary}
\label{cor15}
Under the assumptions of Theorem~\ref{thm51},  suppose that
$\Delta_1 \subset \Delta$  is an infinite set such that
$\lim_{n\in \Delta_1}  \left (\sqrt{\beta^-_n (z_n^*)}/\sqrt{\beta^+_n (z_n^*)} \right) =1.$
Then, for Hill operators,
\begin{equation}
\label{14.11}
\lambda_n^+ - \mu_n \sim \gamma_n \sim 2 \beta_n^+ (z_n^*), \quad
\mu_n - \lambda_n^- = o(\gamma_n), \quad
n \in \Delta_1,
\end{equation}
and for Dirac operators
\begin{equation}
\label{14.110}
\mu_n - \lambda_n^- \sim \gamma_n \sim 2 \beta_n^+ (z_n^*), \quad
\mu_n - \lambda_n^+ = o(\gamma_n), \quad
n \in \Delta_1.
\end{equation}

\end{Corollary}

\begin{proof}
Indeed, by Theorem \ref{thm50} we have
$$
\gamma_n \sim  2\sqrt{\beta_n^-(z_n^*)}\sqrt{\beta_n^+(z_n^*)}
= 2 \beta_n^+ (z_n^*) \cdot
\frac{\sqrt{\beta_n^-(z_n^*)}}{\sqrt{\beta_n^+(z_n^*)}}
\sim  2 \beta_n^+ (z_n^*)  \quad \text{for} \;\; n \in \Delta_1.
$$
In the Hill case, from (\ref{50.4}) in Theorem \ref{thm51} it follows
$$
\lambda_n^+ - \mu_n  \sim \frac{1}{2} \beta_n^+ (z_n^*)
\left (\sqrt{\beta^-_n (z_n^*)}/\sqrt{\beta^+_n (z_n^*)} +1 \right )^2 \sim
2\beta_n^+ (z_n^*) \quad \text{for} \;\; n \in \Delta_1.
$$
Moreover
$$\frac{\mu_n - \lambda_n^-}{\gamma_n}  = \frac{\mu_n - \lambda_n^+}{\gamma_n} +1
\to 0  \quad  \text{as} \;\; n \in \Delta_1, \; n \to \infty,
$$
which completes the proof of (\ref{14.11}).

By the same argument, (\ref{50.5}) implies   (\ref{14.110}).

\end{proof}

The next remark shows that there are wide classes of potentials
for which Corollary~\ref{cor15} is valid.
\begin{Remark}
Let  $v$  be either Hill potential with real nonnegative Fourier coefficients or
Dirac potential of the form $v (x) = \begin{pmatrix}
0 & P(x)  \\P(x)  & 0,
\end{pmatrix}$
where $P$ has real nonnegative Fourier coefficients.
Suppose the assumptions of Theorem~\ref{thm51} hold.
Then  for large enough $n \in \Delta $ we have
\begin{equation}
\label{14.21}
\beta_n^- (z_n^*)= \beta_n^+ (z_n^*) >0,
\end{equation}
and the choice
\begin{equation}
\label{14.22}
\sqrt{\beta_n^- (z_n^*)}= \sqrt{\beta_n^+ (z_n^*)} > 0
\end{equation}
is good for Theorems \ref{thm50} and \ref{thm51}, so (\ref{14.11})
or (\ref{14.110}) hold.  Moreover,  (\ref{14.11}) and  (\ref{14.110})
show that in the Hill case
$\mu_n$ is ``close" to $\lambda_n^-,$
while in Dirac case  $\mu_n$ is ``close" to $\lambda_n^+.$

\end{Remark}

Indeed,  (\ref{14.21}) follows from the
explicit expressions for $\beta^\pm_n (v;z)$  in terms of the Fourier coefficients
of the potential (see \cite{DM15}, formulas (2.16)--(2.33) for Hill
operators, and formulas (2.71)--(2.80) for Dirac operators).

\section{Asymptotics of the spectral gaps
and deviations of Hill  operators with potentials
$v(x) = a e^{-2ix}+ b e^{2ix}$}

1. Consider the Hill potentials of the form
\begin{equation}
\label{i00} v(x) = a e^{-2ix}+ b e^{2ix}, \quad a,b \in
\mathbb{C}\setminus \{0\}.
\end{equation}
This is the most simple two-parameter family of
trigonometric polynomial potentials.

For $b=a \in \mathbb{R}\setminus \{0\}$ we obtain the Mathieu potentials
\begin{equation}
\label{i05} v(x) = 2a \cos 2x, \quad a\neq 0, \;\text{real}.
\end{equation}
For fixed $n$ and $a \to 0, $  Levy and Keller \cite{LK} gave the
following asymptotics for the spectral gaps of the Mathieu operator:
\begin{equation}
\label{i06} \gamma_n = \lambda_n^+ - \lambda_n^- =
\frac{8(|a|/4)^n}{[(n-1)!]^2} \left ( 1+ O(a) \right ), \quad a \to 0.
\end{equation}
Harrell \cite{Har} found, up to a constant factor, the asymptotics
of the spectral gaps of the Mathieu operator for fixed $a$ as  $
n\to \infty. $   Avron and Simon \cite{AS} gave an alternative proof
of  Harrell's asymptotics and found the exact value of the constant
factor, which led to the formula
\begin{equation}
\label{i07} \gamma_n = \lambda_n^+ - \lambda_n^- =
\frac{8(|a|/4)^n}{[(n-1)!]^2} \left ( 1+ o\left (\frac{1}{n^2}\right
) \right ) \quad \text{as} \;\; n \to \infty.
\end{equation}

In \cite[Section 3]{DM25} we studied the
existence of (Riesz) bases consisting of
root functions of the Hill operators with potentials of the
form (\ref{i00}), subject to periodic or antiperiodic boundary
conditions. In this connection,  we found the following asymptotics of
$\beta_n^+ $ and $\beta_n^- $ (see Propositions 5 and 6 in \cite{DM25}):
\begin{equation}
\label{3.22} \beta_n^+ (z) = \frac{4(b/4)^n}{ [(n-1)!]^2} (1+
\eta_n), \quad \eta_n = O(\log n/n)), \;\; |z| \leq 1.
\end{equation}
\begin{equation}
\label{3.24} \beta_n^- (z) = \frac{4(a/4)^n}{ [(n-1)!]^2} (1+
\eta_n), \quad \eta_n = O(\log n/n)), \;\; |z| \leq 1.
\end{equation}
In view of Theorems \ref{thm50} and \ref{thm51},
from these asymptotics we obtain the following.
\begin{Theorem}
\label{thm17}
Consider the Hill operator with a potential of the form (\ref{i00}).
If $|a| \neq |b|,$ then
\begin{equation}
\label{3.31} \gamma_n  \sim   \frac{8\sqrt{a^n}\sqrt{b^n}}{
4^n[(n-1)!]^2} \quad \text{for} \;\; n \in \mathbb{N},
\end{equation}
\begin{equation}
\label{i11} \mu_n - \lambda_n^+  \sim  - \frac{2}{4^n [(n-1)!]^2}
\left (a^n + b^n + 2\sqrt{a^n} \sqrt{b^n}  \right )  \quad
\text{for} \;\; n \in \mathbb{N},
\end{equation}
\begin{equation}
\label{i12} \mu_n - \lambda_n^- \sim  - \frac{2}{4^n [(n-1)!]^2}
\left (a^n + b^n - 2\sqrt{a^n} \sqrt{b^n} \right )  \quad \text{for}
\;\; n \in \mathbb{N},
\end{equation}
where  $\sqrt{a^n} $ and $ \sqrt{b^n}$ are chosen appropriately,
and
\begin{equation}
\label{i13} \mu_n - \frac{1}{2}(\lambda_n^- + \lambda_n^+)
 \sim  - \frac{2}{4^n [(n-1)!]^2} \left (a^n + b^n \right )  \quad \text{for}
\;\; n \in \mathbb{N}.
\end{equation}
\end{Theorem}

\begin{proof}
Since $|a|\neq |b|,$  $\; -1 $ and  $1$ are not cluster points of
the sequence $ \left (\sqrt{a^n}/\sqrt{b^n} \right )_{n\in
\mathbb{N}}$ and $1$ is not a cluster point of $ \left ({a^n}/{b^n}
\right )_{n\in \mathbb{N}}.$  Hence, (\ref{i11}), (\ref{i12}) and
(\ref{i13}) follow from (\ref{50.4}), (\ref{50.5}) and (\ref{50.8})
respectively. 
\end{proof}

The situation is more complicated if $|a|=|b|; $ in general, then we
can claim the validity of (\ref{i11})--(\ref{i13}) only on suitable
subsequences of indices $n.$ The following simple example
illustrates the choice of suitable subsequences, and also the choice
of branches $\sqrt{\beta_n^\pm (z)}.$

\begin{Example}
(i)  Let $b>0 $ and $ a= -b.$ Then $a^n = (-1)^n b^n,$ and we fix
$$
\sqrt{b^n} = b^{n/2}, \quad  \sqrt{a^n}= \begin{cases} b^{n/2}&
\text{for even} \;n, \\  \pm i\, b^{n/2} & \text{for odd}\;
\end{cases}
$$
Then, with a suitable choice of  signs $\pm $ for odd $n,$
we have that
(\ref{i11})
holds for $n \in \mathbb{N},$  (\ref{i12}) holds for $n \in
2\mathbb{N}-1,$ and (\ref{i13}) holds $n \in 2\mathbb{N}.$

(ii) Let $b>0$   and $a=b i.$  Then $a^n = i^n b^n, $ and we fix
$$
\sqrt{b^n} = b^{n/2}, \quad  \sqrt{a^n}= \begin{cases} b^{n/2}&
\text{for}\; \;n\in 4\mathbb{N}, \\  b^{n/2} e^{-i\pi/4} &
\text{for} \;
\; n\in 4\mathbb{N}-1, \\ \pm i b^{n/2} & \text{for} \; \; n\in 4\mathbb{N}-2, \\
b^{n/2} e^{i\pi/4} & \text{for} \; \; n\in 4\mathbb{N}-3.
\end{cases}
$$

Then, with a suitable choice of  signs $\pm $ for  $n\in 4\mathbb{N}-2,$
we have that  (\ref{i11})
hold for $n \in \mathbb{N},$ (\ref{i12}) hold for $n \in \mathbb{N}
\setminus 4\mathbb{N},$ and (\ref{i13}) hold for $n \in \mathbb{N}
\setminus (4\mathbb{N}-2).$
\end{Example}

2. The formulas (\ref{3.22}) and (\ref{3.24}) were good enough for
our goals in \cite[Theorem 7]{DM25}; they allowed us (using
Criterion~\ref{crit2}) to conclude that the periodic (or
antiperiodic) root function system contains a Riesz basis if and
only if $|a|=|b|.$ But a more careful asymptotic analysis could give
more precise estimates of the remainder $\eta_n$ in (\ref{3.22}) and
(\ref{3.24}).

Anahtarci and Djakov \cite{AD1} refined  the Harrell-Avron-Simon
asymptotics (\ref{i07}) as follows: for fixed $a\in \mathbb{C}, \,
a\neq 0,$
\begin{equation}
\label{i90} \lambda_n^+ - \lambda_n^-= \pm
\frac{8(a/4)^n}{[(n-1)!]^2} \left[1 - \frac{a^2}{4n^3}+ O \left(
\frac{1}{n^4}\right)\right], \quad n\rightarrow\infty.
\end{equation}
The same argument could be used (with slight modifications) in
order to obtain
\begin{equation}
\label{i100} \lambda_n^+ - \lambda_n^-= \pm
\frac{8(\sqrt{ab}/4)^n}{[(n-1)!]^2} \left[1 - \frac{ab}{4n^3}+ O
\left( \frac{1}{n^4}\right)\right], \quad n\rightarrow\infty
\end{equation}
in the case of potentials (\ref{i00}).
\bigskip

Let $H_t (a,b)$ denotes the Hill operator
with a potential (\ref{i00}), subject to the
boundary conditions
\begin{equation}
\label{i20} y(\pi) = e^{it} y(0), \quad  y^\prime (\pi) = e^{it}
y^\prime (0), \quad -\pi < t \leq \pi.
\end{equation}
Veliev \cite[Theorem 1]{V}  showed that the operators $H_t (a,b)$
have the following isospectral property:
\begin{equation}
\label{i21} Sp \,(H_t (a,b))= Sp \,(H_t (c,d)) \quad \text{if} \quad
ab = c d,
\end{equation}
where $Sp \,(H_t (a,b))$ denotes the spectrum of the operator $H_t
(a,b).$ Therefore, (\ref{i90}) with $\sqrt{ab} $ instead of $a$
implies directly (\ref{i100}).

Theorem~1 in \cite{V} is a partial case of the following assertion.
\begin{Proposition}
\label{Iso} Let $L_t (v) $ be the Hill operator generated by the
boundary conditions (\ref{i20}), and let $v(z) $ be a $\pi$-periodic
entire function. Set $ v_\zeta (z) = v(z+ \zeta) $ for $\zeta \in
\mathbb{C}. $ Then the spectra of the operators $L_t (v) $ and $L_t
(v_\zeta) $ coincide for every $\zeta \in \mathbb{C}. $
\end{Proposition}

\begin{proof} Fix $\lambda \in Sp\, (L_t (v)) $ and let $y(x)$
be a corresponding eigenfunction. In view of
(\ref{i20}) and the existence-uniqueness theorem for linear ordinary
differential equations  we have
\begin{equation}
\label{i23} y(\pi+x) = e^{it} y(x), \quad x \in \mathbb{R}.
\end{equation}

Since the potential $v(z) $ is an entire function, it is well-known
(e.g., see \cite[Section 3.7]{CL}) that $y(x) $ has an analytic
extension $y(z)$ to an entire function that satisfies the equation
$$y^{\prime \prime}(z) + v(z) y (z) = \lambda y(z), \quad z\in \mathbb{C}. $$
Of course, (\ref{i23}) extends on $\mathbb{C}$ as well, that is
\begin{equation}
\label{i230} y(\pi+z) = e^{it} y(z), \quad z \in \mathbb{C}.
\end{equation}

Fix $\zeta \in \mathbb{C}.$  Then the function $y_\zeta (z) =
y(z+\zeta) $ is a solution of the equation
$$y_\zeta^{\prime \prime}(z) + v_\zeta (z) y_\zeta (z) =
\lambda \, y_\zeta (z), \quad z\in \mathbb{C}. $$ Moreover, by
(\ref{i230}) with $z=\zeta $  we have that
$$    y_\zeta (\pi) = y(\pi+\zeta) = e^{it} y(\zeta) = e^{it} y_\zeta
(0),  \quad  y^\prime_\zeta (\pi) = e^{it} y^\prime (\zeta) = e^{it}
y^\prime_\zeta (0),$$ so $\lambda $ is an eigenvalue of $L_t
(v_\zeta)$ and $y_\zeta $ is a corresponding eigenfunction.
Therefore, we conclude that
$$
Sp \, (L_t (v) ) \subset Sp \, (L_t (v_\zeta) ).
$$
Since the same argument proves the opposite inclusion, this
completes the proof of Proposition~\ref{Iso}.

\end{proof}

 \begin{Corollary}
In the above notations, if $ab=cd,$   then (\ref{i21}) holds.
\end{Corollary}

\begin{proof}
Indeed, let  $v(z) = a e^{-2iz} + b e^{2iz}$
and $ab= cd.$ Choose
$\zeta \in \mathbb{C} $ so that
$$
e^{2i\zeta} = \frac{a}{c} = \frac{d}{b}.
$$
Then it follows that
$$
v(x+\zeta) = a e^{-2i\zeta}e^{-2ix} + b e^{2i\zeta}e^{2ix}=
ce^{-2ix} + d e^{2ix}. $$
\end{proof}

\bigskip

Let us mention, that one can prove Proposition~\ref{Iso} using
the corresponding Lyapunov characteristic function. Recall that
if $\varphi (x, \lambda, v), \, \psi (x, \lambda, v)  $ is a pair of
solutions of the equation
$$
-y^{\prime \prime} + v(x) y = \lambda y
$$
such that
$
\varphi (0, \lambda, v) = 1, \; \varphi^\prime (0, \lambda, v) = 1,
\; \psi (0, \lambda, v) = 0, \; \psi^\prime (0, \lambda, v) = 1,
$
then
$$
D(\lambda, v) = \varphi (\pi, \lambda, v) + \psi^\prime (\pi, \lambda, v)
$$
is the Lyapunov characteristic function, and
the spectrum of the operator $L_t (v)$ is given by the
formula
$$
Sp \, (L_t (v)) = \{\lambda: \; D(\lambda, v) = 2 \cos t\}.
$$

Since $\varphi (\pi, \lambda, v) $ and $ \psi^\prime (\pi, \lambda,
v)$ depend analytically on $\lambda \in \mathbb{C} $ and $v \in L^2
([0,\pi]),$ it follows that $D(\lambda, v) $ is an entire function of
$\lambda $ and depends analytically on $v \in L^2 ([0,\pi]).$
\begin{Proposition}
In the above notations, if  $v(z) $ is a $\pi$-periodic entire
function, then
\begin{equation}
\label{i25} D (\lambda, v_\zeta) = D (\lambda, v), \quad \zeta \in
\mathbb{C}.
\end{equation}
\end{Proposition}

{\em Sketch of a proof.} If $s\in \mathbb{R} $ then a standard
argument from the Floquet theory shows that $D (\lambda, v_s) = D
(\lambda, v).$

Fix $\lambda \in \mathbb{C} $ and consider the function
$$ g(\zeta) = D(\lambda, v_{\zeta}), \quad \zeta \in \mathbb{C}. $$
Since $ D (\lambda, v)$ depends analytically on $v,$ we infer that
$g(\zeta) $ is an entire function. But $g(s) = g(0) $ for $s\in
\mathbb{R},$ thus $g(\zeta) = g(0) $ for $\zeta \in \mathbb{C}, $
that is (\ref{i25}) holds.

Of course, the above remarks have their analogs for Dirac operators as well.
Let $L(v)$ be the Dirac operator (\ref{i2}),  and let $D(v, \lambda )$
be the corresponding Hill-Lyapunov characteristic function.
We denote by $L_t (v) $  the same Dirac operator (\ref{i2})
if  considered with the boundary conditions
 $$
 y(\pi) = e^{it} y(0), \quad - \pi < t \leq \pi.
 $$
\begin{Proposition}
If $P(z) $ and $Q(z)$ are $\pi$-periodic entire functions
and
$v(x) =\begin{pmatrix}
 0  &  P(x)  \\ Q(x) & 0 \end{pmatrix}, $
then
\begin{equation}
\label{ii25} D (\lambda, v_\zeta) = D (\lambda, v), \quad   \forall  \zeta \in
\mathbb{C};
\end{equation}
\begin{equation}
\label{ii26}
Sp \, (L_t (v_\zeta)) = Sp \, (L_t(v)), \quad \zeta \in
\mathbb{C}, \;\; - \pi < t \leq \pi.
\end{equation}
\end{Proposition}

As in Hill case, (\ref{ii25}) follows from Floquet theory,  and (\ref{ii26}) follows from (\ref{ii25})
as $Sp \, (L_t (v)) = \{\lambda: \; D(\lambda, v)= 2 \cos t\}. $

\section{Other two exponential term potentials}

1. Consider Hill potentials of the form
\begin{equation}
\label{6.1.1} v(x) =a e^{-2Rix} +b e^{2Six}, \quad a, b \neq 0,
\end{equation}
with $R, S \in \mathbb{N}, \; R \neq S. $ Then
\begin{equation}
\label{6.1.2}R=dr, \quad S=ds, \;\;
\text{where} \; \; r, s \;\;
\text{are coprime}.
\end{equation}

In \cite{DM27} we analyzed the asymptotics of $\beta_n^\pm (z)$ for
\begin{equation}
\label{6.1.3} n \in \Delta := (rsd) \mathbb{N}, \quad \text{i.e.,}
\;\; n= rsdm, \;\; m \in \mathbb{N},
\end{equation}
and explained (see Theorem~11 there) that

(i) there is no Riesz basis consisting of root functions of the
operator $L_{Per^+} (v);$

(ii) if $R$ and $S$ are odd, then there is no Riesz basis consisting
of root functions of the operator $L_{Per^-} (v).$

Moreover, in the proof of  \cite[Theorem 11]{DM27} we have established (but not formulated explicitly)
the following asymptotics (see \cite[Lemma 6 and (51)]{DM27}).

\begin{Proposition}
Let the functional $\beta_n^\pm $ correspond to
the Hill operator with potentials (\ref{6.1.1}).
Then we have, with the notations from (\ref{6.1.2}) and  (\ref{6.1.3}),
\begin{equation}
\label{6.1.4} \beta^+_n (z) \sim 4s^2d^2 \left ( \frac{b}{4s^2d^2}
\right )^{rm} ((rm-1)!)^{-2}, \quad n=rsdm \in \Delta,
\end{equation}
and
\begin{equation}
\label{6.1.6} \beta^-_n (z) \sim 4r^2d^2 \left ( \frac{a}{4r^2d^2}
\right )^{sm} ((sm-1)!)^{-2}, \quad n=rsdm \in \Delta.
\end{equation}
\end{Proposition}
In view of Theorems \ref{thm50} and \ref{thm51},  the following holds.
\begin{Theorem}
Consider the Hill operator with potentials (\ref{6.1.1}).
Then, for large enough $n=mdsr, \; m \in \mathbb{N},$
\begin{equation}
\label{6.1.10} \gamma_n  \sim \frac{8srd^2
\sqrt{b^{rm}}\sqrt{a^{sm}}}{(2d)^{(r+s)m}s^{rm} r^{sm}(rm-1)!
(sm-1)!}
\end{equation}
and
\begin{equation}
\label{6.1.11} \mu_n - \lambda_n^\pm    \sim -\frac{1}{2} \left (
\frac{2sd \sqrt{b^{rm}}}{(2sd)^{rm}(rm-1)! } \pm \frac{2rd
\sqrt{a^{sm}}}{(2rd)^{sm}(sm-1)! } \right )^2,
\end{equation}
where $ \sqrt{b^{rm}}$ and $\sqrt{a^{sm}}$ are chosen appropriately.
Moreover,
\begin{equation}
\label{6.1.12} \mu_n -\frac{1}{2} ( \lambda_n^- +\lambda_n^+) \sim
-\frac{1}{2} \left (\frac{4s^2 d^2 b^{rm}}{(4s^2d^2)^{rm}[(rm-1)!]^2
} + \frac{4r^2d^2 a^{sm}}{(4r^2d^2)^{sm}[(sm-1)!]^2 } \right ).
\end{equation}

\end{Theorem}

\bigskip

2.  Next  we consider only potentials of the form
$$
v(x)=a e^{-2ix} +b e^{2six}, \; s>2,
$$
and the set of indices
$$
\Delta = \{n= sm-1, \; m \in \mathbb{N}\}.
$$
If $s$ is even, then $\Delta $ consists of odd numbers, and if $s$ is odd then
$\Delta \cap 2\mathbb{N} \neq \emptyset$
and $\Delta \cap (2\mathbb{N} -1) \neq \emptyset.$

The following asymptotics for $\beta_n^\pm (z), \; n \in \Delta $
are established in \cite[Propositions 16 and 17]{DM27}.

\begin{Proposition} Under the above assumption,
for $n= sm-1, $
\begin{equation}
\label{6.2.1}  \beta^+_n (z) \sim \frac{-2sb^m}{(2s)^{2m} m!}
\cdot \frac{\Gamma^2 (1-\frac{1}{s} )\Gamma (m-\frac{2}{s})}
{\Gamma^2 (m-\frac{1}{s}) \Gamma (1-\frac{2}{s})},
\end{equation}
\begin{equation}
\label{6.2.2}  \beta^-_n (z) \sim \frac{a^n}{4^{n-1} [(n-1)!]^2}.
\end{equation}
\end{Proposition}
Applying Theorems \ref{thm50} and \ref{thm51}, we obtain the following.

\begin{Theorem}
Consider the Hill operator with potential
$v(x)=a e^{-2ix} +b e^{2six}, \; s>2. $  Then, for $n=sm-1,$
\begin{equation}
\label{6.2.10}
\gamma_n  \sim  2
\frac{\Gamma (1-\frac{1}{s} )\Gamma^{1/2} (m-\frac{2}{s})}
{\Gamma (m-\frac{1}{s}) \Gamma^{1/2} (1-\frac{2}{s})} \cdot
\frac{(2s)^{1/2} \sqrt{-ab^m}}{(2s)^m (m!)^{1/2}}  \cdot \frac{\sqrt{a^n}}{2^{n-1}(n-1)!}
\end{equation}
and
\begin{equation}
\label{6.2.11}  \mu_n - \lambda_n^\pm \sim
-\frac{1}{2} \left ( \frac{\Gamma (1-\frac{1}{s} )\Gamma^{1/2} (m-\frac{2}{s})}
{\Gamma (m-\frac{1}{s}) \Gamma^{1/2} (1-\frac{2}{s})} \cdot
\frac{(2s)^{1/2} \sqrt{-ab^m} }{(2s)^m (m!)^{1/2}}
\pm
\frac{\sqrt{a^n}}{2^{n-1}(n-1)!} \right )^2,
\end{equation}
where $\sqrt{-ab^m}$ and $\sqrt{a^n} $ are  chosen appropriately.
Moreover
\begin{equation}
\label{6.2.12}  \mu_n -\frac{1}{2} ( \lambda_n^- +\lambda_n^+) \sim
\frac{sb^m}{(2s)^{2m} m!} \frac{\Gamma^2 (1-\frac{1}{s} )\Gamma (m-\frac{2}{s})}
{\Gamma^2 (m-\frac{1}{s}) \Gamma (1-\frac{2}{s})} - \frac{2a^n}{4^{n} [(n-1)!]^2}.
\end{equation}

\end{Theorem}

\section{ Asymptotics for Hill operators with potentials
$v(x) = a e^{-2ix} + b e^{2ix} + A e^{-4ix} + B e^{4ix}$}

When analyzing  potentials  of the form
\begin{equation}
\label{5.1} v= a e^{-2ix} + b e^{2ix} +A e^{-4ix} + B e^{4ix}, \quad
a,b,A,B \neq 0,
\end{equation}
it is convenient to change the parameters in (\ref{5.1}) by setting
 \begin{equation}
\label{5.10} A= -\alpha^2, \quad a= -2\tau \alpha, \quad B = -
\beta^2, \quad b= -2\sigma \beta.
\end{equation}
In these notations the following holds (see Proposition 20 in
\cite{DM25}, the self-adjoint case was done in  \cite{DM10} ).
\begin{Proposition}
\label{prop5.6} If $\tau $ and $ \sigma $  are not odd integers then
for even $n\to \infty$
\begin{equation} \label{5.51}
\beta_n^+ (z)=
\frac{4(i\beta/2)^n}{((n-2)!!)^2} \, \cos \left ( \frac{\pi
\sigma}{2}
 \right ) \left( 1+O \left(\frac{\log n}{n}  \right) \right ),
\quad |z| \leq 1,
\end{equation}
\begin{equation}
\label{5.51a} \beta_n^- (z)= \frac{4(i\alpha/2)^n}{((n-2)!!)^2} \,
\cos \left ( \frac{\pi \tau}{2}
 \right ) \left( 1+O \left(\frac{\log n}{n}  \right) \right ),
 \quad |z| \leq 1,
\end{equation}
and if $\tau $ and $ \sigma $  are not even integers then for odd
$n\to \infty$
\begin{equation}
\label{5.52} \beta_n^+ (z) = \frac{4i(i\beta/2)^n}{((n-2)!!)^2}
\;\frac{2}{\pi} \sin \left ( \frac{\pi \sigma}{2}  \right ) \left(
1+O \left(\frac{\log n}{n}  \right) \right ), \quad |z| \leq 1,
\end{equation}
\begin{equation}
\label{5.52a} \beta_n^- (z)= \frac{4i(i\alpha/2)^n}{((n-2)!!)^2}
\;\frac{2}{\pi} \sin \left ( \frac{\pi \tau}{2}  \right ) \left( 1+O
\left(\frac{\log n}{n}  \right) \right ), \quad |z| \leq 1.
\end{equation}
\end{Proposition}

Theorems \ref{thm50} and \ref{thm51} imply the
following.

\begin{Theorem}
Let $L (v)$ be the Hill operator with a potential $v$ given by
(\ref{5.1}), and let $\alpha, \beta, \tau, \sigma $ are defined by
(\ref{5.10}). Suppose that $|A|\neq |B|$ (i.e., $|\alpha| \neq
|\beta|$).

(a) If $\tau $ and $ \sigma $  are not odd integers then for even
$n\to \infty $
\begin{equation}
\label{5.70} \gamma_n \sim  \frac{8 }{2^n((n-2)!!)^2}
\sqrt{(i\beta)^n \cos \left ( \frac{\pi \sigma}{2} \right )}
\sqrt{(i\alpha)^n\cos \left ( \frac{\pi \tau}{2} \right )}\,,
\end{equation}
\begin{equation}
\label{5.71} \mu_n - \lambda_n^\pm  \sim  \frac{-2}{2^n((n-2)!!)^2}
\left (\sqrt{(i\beta)^n \cos \left ( \frac{\pi \sigma}{2} \right )}
\pm \sqrt{(i\alpha)^n\cos \left ( \frac{\pi \tau}{2} \right )}\,
\right )^2,
\end{equation}
where $\sqrt{(i\beta)^n \cos \left ( \frac{\pi \sigma}{2} \right )}$
and $ \sqrt{(i\alpha)^n\cos \left ( \frac{\pi \tau}{2} \right )} $
are chosen  appropriately.
Moreover,
\begin{equation}
\label{5.72} \mu_n -\frac{1}{2}(\lambda_n^+ + \lambda_n^-) \sim
\frac{-2}{2^n((n-2)!!)^2} \left [(i\beta)^n \cos \left ( \frac{\pi
\sigma}{2} \right ) + (i\alpha)^n\cos \left ( \frac{\pi \tau}{2}
\right )\, \right ]
\end{equation}

(b) If $\tau $ and $ \sigma $  are not even integers then for odd
$n\to \infty $
\begin{equation}
\label{5.80} \gamma_n \sim \frac{2}{\pi}\cdot \frac{8
}{2^n((n-2)!!)^2} \sqrt{i(i\beta)^n \sin \left ( \frac{\pi
\sigma}{2} \right )} \sqrt{i(i\alpha)^n \sin \left ( \frac{\pi
\tau}{2} \right )},
\end{equation}
\begin{equation}
\label{5.81} \mu_n -\lambda_n^\pm \sim -\frac{2}{\pi}\cdot \frac{2
}{2^n((n-2)!!)^2} \left (\sqrt{i(i\beta)^n \sin \left ( \frac{\pi
\sigma}{2} \right )} \pm \sqrt{i(i\alpha)^n \sin \left ( \frac{\pi
\tau}{2} \right )}   \right )^2,
\end{equation}
where $\sqrt{i(i\beta)^n \sin \left ( \frac{\pi \sigma}{2} \right
)}$ and $\sqrt{i(i\alpha)^n \sin \left ( \frac{\pi \tau}{2} \right
)}$ are chosen appropriately.
 Moreover,
\begin{equation}
\label{5.82} \mu_n -\frac{1}{2}(\lambda_n^+ + \lambda_n^-) \sim
-\frac{2}{\pi}\cdot \frac{2 }{2^n((n-2)!!)^2} \left [i(i\beta)^n
\sin \left ( \frac{\pi \sigma}{2} \right ) + i(i\alpha)^n \sin \left
( \frac{\pi \tau}{2} \right ) \right ].
\end{equation}
\end{Theorem}

\section{Dirac operators with potentials
that are trigonometric polynomials}

Consider the Dirac operator with potentials
$v(x)= \begin{pmatrix}  0  &  P(x) \\ Q(x) &  0  \end{pmatrix} $ of the form
\begin{equation}
\label{7.0}
 P(x) =a e^{-2ix} + A e^{2ix}, \quad
Q(x) = b e^{-2ix} + B e^{2ix} , \quad
a, A, b, B \in \mathbb{C}\setminus \{0\}.
\end{equation}
In \cite{DM8,DM9} we have studied these operators in the self-adjoint case
 when $b=\overline{a}  $  and $B=\overline{A}. $
 In particular, we have proved
that all even spectral gaps vanish and established asymptotic formulas
for the odd spectral gaps.  The same approach leads to the following.

\begin{Proposition}
Consider the Dirac operator with potential (\ref{7.0}). Then the corresponding
functionals $\beta_n^\pm (z) $  are well defined for $|n| > N^*$  (where $N^* $
depends on $v$), and
\begin{equation}
\label{7.00}
 \beta_n^\pm (z) \equiv 0  \quad \text{for} \;\; n \in  2\mathbb{Z},
\end{equation}
\begin{equation}
\label{7.1}
\beta^+_n (z) \sim   \begin{cases}
\frac{A^m B^{m+1}}{4^{2m} (m!)^2}   &    \text{for}
\; n=2m-1, \; m \in \mathbb{N},\\
 \frac{a^m b^{m+1}}{4^{2m} (m!)^2}   &    \text{for}
 \; n=-(2m-1), \; m \in \mathbb{N},
\end{cases}
\end{equation}
\begin{equation}
\label{7.2}
\beta^-_n (z) \sim   \begin{cases}
\frac{a^{m+1} b^{m}}{4^{2m} (m!)^2}   &    \text{for}
\; n=2m+1, \; m \in \mathbb{N},\\
 \frac{A^{m+1} B^{m}}{4^{2m} (m!)^2}   &    \text{for}
 \; n=-(2m+1), \; m \in \mathbb{N}.
\end{cases}
\end{equation}

\end{Proposition}

Now Theorems \ref{thm50} and  \ref{thm51} imply the
following.
\begin{Theorem}
Consider the Dirac operator with potential (\ref{7.0}). Then
\begin{equation}
\label{7.10}
\gamma_n  \sim   \begin{cases}   \frac{2\sqrt{A^m B^{m+1}} \sqrt{a^{m+1} b^m}}
{4^{2m} (m!)^2}   &    \text{for}  \; n=2m+1, \; m \in \mathbb{N},\\
 \frac{2\sqrt{a^m b^{m+1}}\sqrt{A^{m+1}B^m}}{4^{2m} (m!)^2}   &
 \text{for}  \; n=-(2m+1), \; m \in \mathbb{N},
\end{cases}
\end{equation}
\begin{equation}
\label{7.11}
\mu_n - \lambda_n^\pm
\sim   \begin{cases}   -\frac{1}{2} \left (\frac{\sqrt{A^m B^{m+1}} }
{4^{m} m!}    \pm  \frac{\sqrt{a^{m+1} b^m} }
{4^{m} m!}
   \right )^2   &    \text{for}  \; n=2m+1, \; m \in \mathbb{N},\\
\frac{1}{2} \left ( \frac{\sqrt{a^m b^{m+1}}}{4^{2m} (m!)^2} \pm
\frac{\sqrt{A^{m+1}B^m}}{4^{2m} (m!)^2}  \right )^2 &
\text{for}  \; n=-(2m+1), \; m \in \mathbb{N},
\end{cases}
\end{equation}
where $\sqrt{A^m B^{m+1}}$  and $ \sqrt{A^{m+1}B^m}$
are chosen appropriately.
Moreover,
\begin{equation}
\label{7.12}
 \mu_n -\frac{1}{2}(\lambda_n^+ + \lambda_n^-)
\sim   \begin{cases}   \frac{1}{2} \left (\frac{A^m B^{m+1} }
{4^{2m} (m!)^2}    +  \frac{a^{m+1} b^m}
{4^{2m} (m!)^2}
   \right )   &    \text{for}  \; n=2m+1, \; m \in \mathbb{N},\\
\frac{1}{2} \left ( \frac{a^m b^{m+1}}{4^{2m} (m!)^2} +
\frac{A^{m+1}B^m}{4^{2m} (m!)^2}  \right )&
\text{for}  \; n=-(2m+1), \; m \in \mathbb{N}.
\end{cases}
\end{equation}
\end{Theorem}

\section{Concluding Remarks and Comments}

1. Theorem \ref{thm50} claims the existence of branches of
$\sqrt{\beta_n^- (z)} $ and $\sqrt{\beta_n^+ (z)} $
such that  (\ref{50.63}), (\ref{50.64}) and (\ref{50.3}) hold.
The same choice of branches is good for the validity of
Theorem~\ref{thm51}.

When applying these theorems (see Section 5-8)
the asymptotics of $\beta_n^\pm $ are known, say we have
\begin{equation}
\label{10.1}
\beta_n^\pm (z) \sim B_n^\pm \quad  \text{for} \quad |z| \leq \rho_n,
\end{equation}
where the sequences  $B_n^\pm $ are
given by explicit expressions. Of course, then one could
replace formally $\sqrt{\beta_n^\pm (z_n^*)}$ by  $\sqrt{B_n^\pm} $ in
the asymptotic formulas (\ref{50.3}) and (\ref{50.4})--(\ref{50.9}).
The following question arise: What choice of these square
roots would guarantee the validity of those formulas?

It turns out that sometimes we cannot make a canonical choice.
The difficulty of choosing explicitly those square roots stems out from
our definition of $\lambda_n^\pm.$  Recall that
$\lambda_n^+$  is the eigenvalue which real part is larger, or which imaginary part
is larger if the real parts are equal.  In other words, by definition for
the difference $\gamma_n = \lambda_n^+ - \lambda_n^-$
we have either $Re \, \gamma_n  > 0$ or $Re \, \gamma_n  =0$
but $Im \, \gamma_n >0.$   This definition is asymptotically unstable, so if
we know the asymptotics
$$\gamma_n \sim  \pm i \, H_n,  \quad H_n >0 $$
only, then it is impossible to determine the correct sign.

\begin{Remark}
\label{rem30}
Let (\ref{10.1}) holds, and let
there is a constant $\varepsilon_0 > 0 $  such that
\begin{equation}
\label{20.1}
- \pi + \varepsilon_0 <  Arg  \left ( B^-_n B_n^+
) \right ) < \pi - \varepsilon_0,
\quad \text{for} \quad n \in \Delta.
\end{equation}
Then formulas (\ref{50.3}) and (\ref{50.4})--(\ref{50.9})
hold, with $\sqrt{\beta_n^\pm (z_n^*)}$ replaced by
$\sqrt{B_n^\pm}, $
if and only if $\sqrt{B_n^-}$
and $\sqrt{B_n^+}$  are chosen so that
\begin{equation}
\label{20.2}
Re  \left ( \sqrt{B_n^-} \sqrt{ B_n^+} \right )  >0
\quad \text{for} \quad n \in \Delta.
\end{equation}
\end{Remark}

Next we consider the choice of square roots in the context of
Theorem~\ref{thm17}.
\begin{Question}
\label{q1}
Let $a=|a|e^{i\varphi} $ and $  b= |b|e^{i\psi}; $  then
\begin{equation}
\label{20.20}
\sqrt{a^n} = \pm |a|^{n/2}  e^{in\varphi/2}, \quad
\sqrt{b^n} = \pm |b|^{n/2}  e^{in\psi/2}.
\end{equation}
Can one give a canonical choice  of 
$\sqrt{a^n} $   and $\sqrt{b^n} $ 
(i.e., a choice of signs in (\ref{20.20}))
that is good for the validity of
Theorem~\ref{thm17}?
\end{Question}

There are cases, when Remark~\ref{rem30} may be used to
answer Question~\ref{q1}. Consider the following.
\begin{Example}
Suppose $a=|a|e^{i\varphi}, \;  b= |b|e^{i\psi} $ and $\varphi +
\psi = \frac{p}{q} \pi, $  where $p, q \in \mathbb{N}$ are
relatively prime and $p $ is even.  Then (\ref{20.1}) holds with
$\varepsilon_0 = \pi/q$ for  $n\in \mathbb{N},$ so
Remark~\ref{rem30} could be applied.
\end{Example}

However, in general  we cannot apply Remark~\ref{rem30}.
For example,  if the number $(\varphi + \psi)/\pi  $  is  irrational,
then   $-1 = e^{i\pi}$
is a cluster point of the sequence $\left (e^{i n (\varphi + \psi)}\right) .$

\bigskip

2.  The crucial assumptions in our Theorems \ref{thm50} and
\ref{thm51} are \eqref{50.1} and  \eqref{50.2}. The examples
considered in Sections 5--8 (where all potentials are trigonometric
polynomials) are rather complicated. We were able to handle  those
examples only because the asymptotics of the corresponding functionals
$\beta_n^\pm (z)$  were known from our earlier papers.

The following  interesting class of Hill potentials was introduced  in 1996 by A. Shkalikov
in Seminar on Spectral Analysis at Moscow State University:
\begin{equation}
\label{w1} v \in W^{m+1}_1, \;\;
v^{(s)} (\pi) = v^{(s)} (0) \;\; \text{for} \;\; s=0, \ldots, m-1, 
\quad  v^{(m)} (\pi) \neq  v^{(m)} (0).
\end{equation}
He conjectured that for such potentials  the periodic (or antiperiodic)
system of normalized root functions is a Riesz basis and 
suggested a scheme of a proof.
The study of this class and its modifications confirmed the Shkalikov conjecture 
and led to a series of results on Riesz basis property \cite{MK, Ma06-1,VS09}.

\begin{Lemma}
\label{lem00}
If $v $ satisfies (\ref{w1}), then
\begin{align}
\label{9.3}
\beta_n^- (z) & \sim V(-2n) \sim
\frac{1/\pi}{(-2in)^{m+1}} \left ( v^{(m)}(0) - v^{(m)}(\pi)   \right ) ,
\quad |z| \leq n,
 \\ \label{9.4}
\beta_n^+ (z) &  \sim V(2n) \sim
\frac{1/\pi}{(2in)^{m+1}} \left ( v^{(m)}(0) - v^{(m)}(\pi)   \right ),
\quad |z| \leq n.
\end{align}
\end{Lemma}

In \cite{DM32}, we study systematically classes of Hill potentials $v$
such that the corresponding functionals $\beta_n^\pm $
satisfy
$$
\beta_n^\pm (z) \sim V(\pm 2n).
$$
Lemma~\ref{lem00}
follows from our results in \cite{DM32}.

In view of Lemma~\ref{lem00},
Theorems \ref{thm50} and \ref{thm51} imply the
following.
\begin{Proposition}
Let  $v$  satisfies (\ref{w1}).  

(a)  If $m$ is odd, then
\begin{equation}
\label{9.11} \gamma_n  \sim \pm \frac{2/\pi}{(2in)^{m+1}} \left (
v^{(m)}(0) - v^{(m)}(\pi)   \right ),
\end{equation}
where the signs $\pm$  are chosen appropriately, and
\begin{equation}
\label{9.13} \mu_n - \frac{1}{2}(\lambda_n^- + \lambda_n^+)
 \sim    \frac{1/\pi}{(2in)^{m+1} }
\left ( v^{(m)}(\pi) - v^{(m)}(0)   \right ).
\end{equation}
Moreover,

(a1)   if   $ \Delta \subset \mathbb{N} $ is an infinite set, and
the sign on the right of \eqref{9.11} is $(+)$ for $n \in \Delta, $
then
\begin{equation}
\label{9.15} \mu_n - \lambda_n^+  \sim \frac{2/\pi}{(2in)^{m+1}}
\left ( v^{(m)}(0) - v^{(m)}(\pi)   \right ) \quad \text{for} \;\;
n\in \Delta;
\end{equation}

(a2)  if   $ \Delta \subset \mathbb{N} $ is an infinite set, and the
sign on the right of \eqref{9.11} is $(-)$ for $n \in \Delta, $ then
\begin{equation}
\label{9.17} \mu_n - \lambda_n^-  \sim - \frac{2/\pi}{(2in)^{m+1}}
\left ( v^{(m)}(0) - v^{(m)}(\pi)   \right ) \quad \text{for} \;\;
n\in \Delta.
\end{equation}

(b) If $m$ is even, then
\begin{equation}
\label{9.12} \gamma_n  \sim \pm \frac{2i/\pi}{(2in)^{m+1}} \left (
v^{(m)}(0) - v^{(m)}(\pi)   \right ),
\end{equation}
where the signs $\pm$  are chosen appropriately,
but there is no formula for $\mu_n - \frac{1}{2}(\lambda_n^- + \lambda_n^+).$

Moreover,
\begin{equation}
\label{9.14} \mu_n - \lambda_n^+  \sim \mp \frac{i/\pi}{(2in)^{m+1}}
\left ( v^{(m)}(0) - v^{(m)}(\pi)   \right ),
\end{equation}
where the signs are opposite to those in (\ref{9.12}), and
\begin{equation}
\label{9.16} \mu_n - \lambda_n^-  \sim \pm \frac{i/\pi}{(2in)^{m+1}}
\left ( v^{(m)}(0) - v^{(m)}(\pi)   \right )
\end{equation}
where the signs are the same as in (\ref{9.12}).

\end{Proposition}

\begin{proof}
If $m$ is odd, then
$$
\beta_n^\pm (z) \sim B_n:= \frac{1/\pi}{(2in)^{m+1}} \left (
v^{(m)}(0) - v^{(m)}(\pi)   \right ).
$$
We have $\beta_n^- (z)/\beta_n^+ (z) \sim 1, $ so (\ref{9.13})
follows from (\ref{500.4}).

Let  $B_n = |B_n| e^{\varphi_n}$  with $\varphi_n \in (-\pi, \pi],$
and let $\sqrt{B_n} = |B_n|^{1/2} e^{\varphi_n/2}. $
We have $\beta_n^\pm \sim B_n $ but nevertheless it may happen
that the appropriate choice of branches $\sqrt{\beta_n^\pm (z)} $ is
such that
$$
\sqrt{\beta_n^- (z)}  \sqrt{\beta_n^+ (z)} \sim - B_n.
$$

In the case (a1) we have $\sqrt{\beta_n^- (z)}/\sqrt{\beta_n^+ (z)}
\sim 1 \quad \text{for} \;\; n \in \Delta,$ so (\ref{9.15}) follows
from (\ref{50.4}).

In the case (a2) we have $\sqrt{\beta_n^- (z)}/\sqrt{\beta_n^+ (z)}
\sim -1 \quad \text{for} \;\; n \in \Delta,$ so (\ref{9.15}) follows
from (\ref{50.8}).

In the case (b) $m$ is even, so we have
$$
\beta_n^- (z) \sim -B_n,\quad \beta_n^+ (z) \sim B_n,
$$
where is given by the same expression as above. Since $\beta_n^-
(z)/\beta_n^+ (z) \sim - 1,$  Theorem~\ref{thm51} does not provide
an asymptotic formula for $\mu_n - \frac{1}{2}(\lambda_n^- +
\lambda_n^+).$

Of course, (\ref{9.11}) holds with appropriate choice of signs
$\pm.$ On the other hand, we have
$$
\beta_n^- (z)+\beta_n^+ (z) = B_n \, o(1),
$$
so from Theorem~\ref{thm51} it follows that
$$
\mu_n - \lambda_n^+ \sim - \gamma_n/2, \quad \mu_n - \lambda_n^-
\sim \gamma_n/2.
$$
Thus, (\ref{9.14}) and (\ref{9.16}) hold.
\end{proof}

\bigskip

3.  The Hill operator
$$
L(v) y = y^{\prime \prime} + v(x) y,  \quad v \in L^1 ([0,\pi])
$$
considered with Neumann boundary conditions
$$
y^\prime (\pi) = y^\prime (0) =0
$$
has a discrete spectrum;  for large enough $n$ there is exactly
one Neumann eigenvalue
$\nu_n $ such that $|\nu_n - n^2| < n. $
Let
$$
\delta_n^{Neu} = \nu_n - \lambda_n^+ \quad \text{and}
\quad \delta_n^{Dir} = \mu_n - \lambda_n^+
$$
be the Neumann and Dirchlet deviations.
Recently Batal \cite{B13} proved, for potentials $v \in L^p  ([0,\pi]),$
that
$$
|\delta^{Neu}_n | + |\gamma_n|    \sim  |\delta^{Dir}_n | + |\gamma_n|.
$$
This asymptotic equivalence shows that one can replace the
Dirichlet deviation by the
Neumann deviation in Criterion~\ref{crit3}  and in \cite[Theorem~67]{DM15}.

Does there exists an analog of Theorem~\ref{thm51} for the Neumann deviations
$\nu_n - \lambda_n^+?$   The answer is positive --
following the same approach as in the proof of Theorem~\ref{thm51}
one can obtain asymptotic formulas for
the Neuman deviations  in terms of $\beta_n^\pm (z_n^*).$

\end{document}